\documentclass{amsart}
\pdfoutput=1
\usepackage{amsfonts,amsmath,amscd,amssymb,amsthm,graphicx}
\usepackage{hyperref}
\usepackage[mathscr]{euscript} 
\usepackage{enumerate,color}
\usepackage{fixltx2e}

\usepackage{algorithm}
\usepackage[noend]{algpseudocode}

\algnewcommand\To{\textbf{to}}
\algnewcommand\Break{\textbf{break}}

\usepackage{mathtools}
\newtheorem{theorem}{Theorem}
\newtheorem{definition}[theorem]{Definition}
\newtheorem{lemma}[theorem]{Lemma}
\newtheorem{assumption}{Assumption}

\renewcommand{\c}[1]{{\mathcal{#1}}}
\newcommand{\mat}[1]{\boldsymbol{#1}}
\newcommand{\randmat}[1]{\mat{{\mathcal{{#1}}}}}
\DeclareMathOperator{\diag}{diag}
\DeclareMathOperator{\blockdiag}{blockdiag}
\newcommand{\MG}{M\!G}
\newcommand{\RR}{\mathbb{R}}
\newcommand{\PP}{\mathbb{P}}
\renewcommand{\SS}{\mathbb{S}}
\newcommand{\NN}{\mathbb{N}}

\DeclareMathOperator{\grad}{grad}

\DeclareMathOperator{\curl}{curl}

\DeclareMathOperator{\Div}{div}

\renewcommand{\Pr}[1]{\mathbb{P}\left(#1\right)}
\newcommand{\Exp}[1]{\mathbb{E}\left[#1\right]}
\newcommand{\Var}[1]{\mathbb{V}\left[#1\right]}
\DeclareMathOperator{\Cov}{\mathrm{Cov}}
\providecommand{\norm}[1]{\left\lVert#1\right\rVert}
\providecommand{\opnorm}[1]{\left\lVert#1\right\rVert_{A^{2}}}
\providecommand{\abs}[1]{\left|#1\right|}

\newcommand\numberthis{\addtocounter{equation}{1}\tag{\theequation}}

\usepackage{cleveref}
\crefformat{equation}{(#2#1#3)}
\Crefname{ALC@unique}{Line}{Lines}
\crefname{assumption}{Assumption}{Assumptions}
\Crefname{subsection}{Section}{Sections}

\usepackage{placeins}
\usepackage{thmtools,thm-restate,cite}

\allowdisplaybreaks

\title[Is the multigrid method fault tolerant? The Multilevel Case.]{Is the multigrid method fault tolerant? \\ The Multilevel Case.}

\thanks{The effort of MA was partially supported by SIRIUS award from DoE.}

\author{Mark Ainsworth}
\address{Division of Applied Mathematics, Brown University, 182 George St, Providence, RI 02912, USA.\newline
  Computer Science and Mathematics Division, Oak Ridge National Laboratory, Oak Ridge, TN 37831, USA.}
\email{Mark\_Ainsworth@Brown.edu}
\author{Christian Glusa}
\address{Division of Applied Mathematics, Brown University, 182 George St, Providence, RI 02912, USA.}
\email{Christian\_Glusa@Brown.edu}

\begin{document}

\begin{abstract}
  Computing at the exascale level is expected to be affected by a significantly higher rate of faults, due to increased component counts as well as power considerations.
  Therefore, current day numerical algorithms need to be reexamined as to determine if they are fault resilient, and which critical operations need to be safeguarded in order to obtain performance that is close to the ideal fault-free method.

    In a previous paper \cite{AinsworthGlusa216_IsMultigridMethodFaultTolerant}, a framework for the analysis of random stationary linear iterations was presented and applied to the two grid method.
  The present work is concerned with the multigrid algorithm for the solution of linear systems of equations, which is widely used on high performance computing systems.
  It is shown that the Fault-Prone Multigrid Method is not resilient, unless the prolongation operation is protected.
  Strategies for fault detection and mitigation as well as protection of the prolongation operation are presented and tested, and a guideline for an optimal choice of parameters is devised.
\end{abstract}

\keywords{Multigrid, Fault Tolerance, Resilience, Random Matrices, Convergence Analysis}
\subjclass{65F10, 65N22, 65N55, 68M15}

\maketitle

\section{Introduction}
\label{sec:introduction}

The anticipated arrival of exascale systems opens up the possibility for scientific simulation at dramatically increased scale, provided that several technical challenges can be overcome.
One such challenge concerns the resilience of the underlying numerical algorithm in the face of the increased level of faults and node failures on an exascale machine \cite{CappelloGeistEtAl2009_TowardExascaleResilience,CappelloGeistEtAl2014_TowardExascaleResilience}.

Faults are classified as \emph{hard} or \emph{soft} \cite{AvizienisLaprieEtAl2004_BasicConceptsTaxonomyDependableSecureComputing}.
While hard faults require immediate remedial action in order for the program execution to proceed, the effect of soft faults is to perturb data and instructions, possibly undetected but having the potential to degrade the performance of the algorithm or even invalidate the entire simulation.
Node failures or lost messages constitute examples of hard faults, whereas random bit flips, induced by cosmic particles, can be classified as soft faults.
It is of urgent concern to determine the effect of faults on existing state of the art numerical methods which will be utilised on exascale systems.
In the cases where the algorithms are found wanting, one would like to know how to modify the schemes to cope with the new reality of fault-prone computing.

The widely used \emph{multigrid algorithm} for the solution of linear systems of equations is the concern of this work.
In earlier work \cite{AinsworthGlusa216_IsMultigridMethodFaultTolerant}, we introduced a simple fault mitigation strategy called \emph{laissez-faire}, whereby, instead of trying to recover any lost or corrupted information, one simply replaces affected values with zero before continuing the execution.
This approach has several advantages in terms of efficiency including no requirement for communication.
However, the acid test is whether or not this strategy is effective and results in any improvement in fault resilience.

The goal of this work is twofold.
Firstly, we extend the analytic convergence estimates from the two grid setting to an arbitrary number of levels.
Secondly, we investigate how the algorithm can be applied in practice where some form of fault detection is needed.
In \Cref{sec:review}, we will recall the basic theoretical framework for fault-prone linear iterative methods as well as some previous results needed in the analysis.
We then state a general convergence estimate for Fault-Prone Multigrid in \Cref{sec:main-results}, and demonstrate its validity for a range of numerical examples.
Finally, the practical issue of fault detection and protection of the involved operations is discussed in \Cref{sec:implementation}.
The effects of different levels of protection are demonstrated for a model problem, and optimal parameters depending on fault rate and problem size are given.
The theoretical details and proofs of the results are collected in the Appendix.

\section{Preliminaries}
\label{sec:review}

\subsection{Modelling Faults}

In \cite{AinsworthGlusa216_IsMultigridMethodFaultTolerant}, we proposed to model the effect of faults and their mitigation in iterative linear solvers through random diagonal matrices.
A fault-free vector $x\in\RR^n$ gets transformed into \(\tilde{x}\) according to
\begin{align*}
  \tilde{x} = \randmat{X}x.
\end{align*}
Throughout an algorithm, several operations might be subject to faults, each corresponding to such a fault matrix.
For given \(\varepsilon>0\), let $\SS_{\varepsilon}$ denote the set of these fault matrices, and assume that they satisfy the following conditions:
\begin{assumption} \label{as:faults}\leavevmode
  \begin{enumerate}

  \item Each $\randmat{X}\in\SS_{\varepsilon}$ is a random diagonal matrix.

  \item For every $\randmat{X}\in\SS_{\varepsilon}$, there holds $\Exp{\randmat{X}} =
    e(\randmat{X})\mat{I}$, where $e(\randmat{X})>0$, and
    $|e(\randmat{X})-1|\le C\varepsilon$ for some fixed $C>0$.

  \item For every
    $\randmat{X}\in \SS_{\varepsilon}$ there holds $\norm{\Var{\randmat{X}]}}_{2}
    = \max_{i,j}\left|\Cov\left(\randmat{X}_{ii},\randmat{X}_{jj}\right)
    \right| \leq\varepsilon$.

  \end{enumerate}
\end{assumption}

\(\varepsilon\) measures how close the fault matrices are to the identity, i.e. the fault-free case.
While the value of \(\varepsilon\) could vary from operation to operation, for example to take into account that denser matrix-vector products take more time and are therefore more likely to be hit by faults, we neglect this aspect in the analysis for the sake of clarity.

Important examples of random faults covered by \Cref{as:faults} are
\begin{enumerate}
\item \emph{Componentwise detectable faults}

  A typical example of a detectable fault would be flipping of individual bits in a floating point number, resulting in a large enough upset to be detected, or a pointer corruption that leads to an invalid memory address.
  Such cases can be treated using the laissez-faire strategy described above, and therefore modelled as componentwise faults with
  \begin{align}
    \randmat{X}= \diag\left(\chi_{1},\dots,\chi_{n}\right),\label{eq:componentwiseFaults}
  \end{align}
  where $\chi_i$ are independent and identically \(\varepsilon\)-distributed Bernoulli random variables, i. e.
  \begin{align*}
    \chi_{i} &=
           \begin{cases}
             0          & \text{with probability } \varepsilon, \\
             1          & \text{with probability } 1-\varepsilon.
           \end{cases}
  \end{align*}

\item \emph{Blockwise detectable faults}

  In the event of a node failure and application of the laissez-faire strategy, all the components of a vector that were residing on the node will be zeroed out.
  This can be modelled by a random matrix with independent diagonal \emph{blocks}:
  \begin{align}
    \randmat{X}= \blockdiag\left(\chi_{1}\mat{I}_{1},\dots,\chi_{N}\mat{I}_{N}\right), \label{eq:blockwiseFaults}
  \end{align}
  where $\chi_i$ are independent and identically \(\varepsilon\)-distributed Bernoulli random variables and \(\mat{I}_{j}\) are identity matrices.

\item \emph{Silent faults}

  Soft faults may be difficult or even impossible to detect, especially if their induced perturbation is small. Such silent faults can be modelled by a random matrix
  \begin{align}
    \randmat{X} = \mat{I}+\diag\left(\eta_1\chi_1, \dots, \eta_n\chi_n\right) \label{eq:silentFaults}
  \end{align}
  where $\eta_i$ are independent and identically distributed random variables, and \(\chi_{i}\) are independent and identically distributed Bernoulli random variables, such that \(\randmat{X}\in \SS_{\varepsilon}\).
  In particular, this in includes the cases of frequent faults with small impact and of rare but large upsets.
\end{enumerate}

\subsection{Multigrid Algorithm}

Let $\mat{A}$ be a symmetric, positive definite matrix arising from a finite element discretization of an elliptic partial differential equation in $d$ spatial dimensions.
The multigrid method solves the linear system of equations
\begin{align*}
  \mat{A}x = b
\end{align*}
by introducing a hierarchy of coarser problems
\begin{align*}
  \mat{A}_{\ell}x_{\ell}=b_{\ell}
\end{align*}
of size \(n_{\ell}\), \(\ell=0,\dots,L\) with \(\mat{A}_{L}=\mat{A}\).
Information gets transferred between levels through restriction and prolongation operators
\begin{align*}
  \mat{R}_\ell^{\ell+1}: \RR^{n_{\ell+1}}\rightarrow \RR^{n_\ell}, \quad
  \mat{P}_{\ell+1}^\ell: \RR^{n_\ell}\rightarrow \RR^{n_{\ell+1}}.
\end{align*}
We will assume that \(\mat{P}_{\ell+1}^\ell=\left(\mat{R}_\ell^{\ell+1}\right)^{T}\) along with the usual Galerkin relation \(\mat{A}_{\ell} = \mat{R}_{\ell}^{\ell+1}\mat{A}_{\ell+1}\mat{P}_{\ell+1}^{\ell}\).
We will drop sub- and superscripts on restriction and prolongation operators in what follows.
Moreover, smootheners are defined on each level as
\begin{align*}
  x_{\ell}\leftarrow x_{\ell}+\mat{N}_{\ell}\left(b_{\ell}-\mat{A}_{\ell}x_{\ell}\right),
\end{align*}
where \(\mat{N}_{\ell}\) is a suitable preconditioner, e.g. the damped Jacobi preconditioner \(\mat{N}_{\ell}=\theta \mat{D}_{\ell}^{-1}\), with \(\mat{D}_{\ell}\) the diagonal part of \(\mat{A}_{\ell}\).

The Fault-Prone Multigrid Method \(\c{\MG}_{\ell}\) was described in detail in \cite{AinsworthGlusa216_IsMultigridMethodFaultTolerant}, and is given by \Cref{alg:ndmg}.
The classical fault-free variant \(\MG_{\ell}\) can be obtained by replacing all fault matrices \(\randmat{X}^{(\bullet)}_{\ell}\) with identity matrices.

\begin{algorithm}
  \begin{algorithmic}[1]
    \Require Right hand side $b_\ell$; Initial iterate $x_\ell$
    \Ensure $\c{\MG}_\ell(b_\ell, x_\ell)$
    \If{$\ell=0$}
    \Return $\mat{A}_{0}^{-1}b_{0}$  \Comment{Exact solve on coarsest grid}
    \Else
    \For{$i\leftarrow 1$ \To{} $\nu_{1}$}
    \State $x_\ell\leftarrow x_{\ell}+\randmat{X}^{(S,\text{pre})}_{\ell}\mat{N}_{\ell}\left(b_\ell-\mat{A}_{\ell}x_\ell\right)$  \Comment{$\nu_1$ pre-smoothing steps}
    \EndFor
    \State $d_{\ell-1}\leftarrow\randmat{X}^{(R)}_{\ell-1}\mat{R}\randmat{X}^{(\rho)}_\ell\left(b_\ell-\mat{A}_\ell x_\ell\right)$  \Comment{Restriction to coarser grid}
    \State $e_{\ell-1}\leftarrow 0$
    \For{$j\leftarrow 1$ \To{} $\gamma$}
    \State $e_{\ell-1}\leftarrow \c{\MG}_{\ell-1}(d_{\ell-1},e_{\ell-1})$  \Comment{$\gamma$ coarse grid correction steps}
    \EndFor
    \State $x_\ell\leftarrow x_\ell + \randmat{X}^{(P)}_\ell\mat{P}e_{\ell-1}$  \Comment{Prolongation to finer grid}
    \For{$i\leftarrow 1$ \To{} $\nu_{2}$}
    \State $x_\ell\leftarrow x_{\ell}+\randmat{X}^{(S,\text{post})}_{\ell}\mat{N}_{\ell}\left(b_\ell-\mat{A}_{\ell}x_\ell\right)$  \Comment{$\nu_2$ post-smoothing steps}
    \EndFor
    \EndIf
  \end{algorithmic}
  \caption{Model for Fault-Prone Multigrid Algorithm $\c{\MG}_\ell$ where $\randmat{X}^{(\bullet)}_\ell$ are random diagonal matrices.}
  \label{alg:ndmg}
\end{algorithm}

The classical approach to the analysis of iterative solution methods for linear systems uses the notion of an iteration matrix.
For the fault-prone method \(\c{\MG}_{\ell}\), it is defined by the equation
\begin{align*}
  x - \c{\MG}_\ell(\mat{A}_\ell x, y) = \randmat{E}_\ell(x-y),
  \quad\forall x,y\in\RR^{n_\ell},
\end{align*}
and given by
\begin{align}\label{eq:Mrecursion}
  \randmat{E}_{0}&=\mat{0}, \nonumber \\
  \randmat{E}_\ell
  &= \left(\randmat{E}^{S,\text{post}}_\ell\right)^{\nu_2}
  \left[ \mat{I}
  - \randmat{X}^{(P)}_\ell\mat{P}( \mat{I} -\randmat{E}^\gamma_{\ell-1} )
  \mat{A}_{\ell-1}^{-1}\randmat{X}^{(R)}_{\ell-1} \mat{R}\randmat{X}^{(\rho)}_\ell \mat{A}_\ell
  \right]
  \left(\randmat{E}^{S,\text{pre}}_\ell\right)^{\nu_1},
\end{align}
for $\ell=1,\ldots,L$, with \(\randmat{E}_{\ell}^{S,\text{pre/post}}=\mat{I}-\randmat{X}_{\ell}^{(S,\text{pre/post})}\mat{N}_{\ell}\mat{A}_{\ell}\).
Here, we have used superscripts pre and post to reflect that the pre- and post-smootheners are \emph{independent realisations of the same random matrix}.
Moreover, we will use powers of random matrices to signify products of independent realisations of the same random matrix.

Setting $\randmat{E}_{L-1}=\mat{0}$ and applying the recursion~\cref{eq:Mrecursion} in the case $\ell=L$ yields a formula for the iteration matrix of the \emph{Fault-Prone Two Grid Algorithm}:
\begin{align}\label{eq:twolevel}
    \randmat{E}^{TG}_L
    &= \left(\randmat{E}^{S,\text{post}}_L\right)^{\nu_2}
      \left[ \mat{I} - \randmat{X}^{(P)}_L\mat{P}
	  \mat{A}_{L-1}^{-1}\randmat{X}^{(R)}_{L-1} \mat{R}\randmat{X}^{(\rho)}_{L} \mat{A}_L
    \right]
    \left(\randmat{E}^{S,\text{pre}}_L\right)^{\nu_1} \\
  &=\left(\randmat{E}^{S,\text{post}}_L\right)^{\nu_2}
  \randmat{E}_{L}^{CG}
  \left(\randmat{E}^{S,\text{pre}}_L\right)^{\nu_1},\nonumber
\end{align}
corresponding to using an exact solver on level $L-1$.
Here \(\randmat{E}^{CG}_{\ell}\) is the iteration matrix of the exact fault-prone coarse grid correction on level \(\ell\).

By replacing all fault matrices with the identity, the classical fault-free quantities are recovered:
\begin{align*}
  \mat{E}_{\ell} &= \left(\mat{E}_{\ell}^{S}\right)^{\nu_{2}} \left[\mat{I}-\mat{P}\left(\mat{I}-\mat{E}_{\ell-1}^{\gamma}\right)\mat{A}_{\ell}^{-1}\mat{R}\mat{A}_{\ell}\right] \left(\mat{E}_{\ell}^{S}\right)^{\nu_{1}}, \\
  \mat{E}_{\ell}^{S}&=\mat{I}-\mat{N}_{\ell}\mat{A}_{\ell}, \\
  \mat{E}_{\ell}^{TG} &= \left(\mat{E}_{\ell}^{S}\right)^{\nu_{2}} \mat{E}^{CG}_{\ell} \left(\mat{E}_{\ell}^{S}\right)^{\nu_{1}}, \\
  \mat{E}^{CG}_{\ell}&=\mat{I}-\mat{P}\mat{A}_{\ell}^{-1}\mat{R}\mat{A}_{\ell}.
\end{align*}
These are the iteration matrices of the fault-free multigrid methods, smoothener, two grid method and coarse-grid correction.

\subsection{Convergence Theory for Standard Multigrid}
\label{sec:conv-theory-fault-free}

The convergence proof of the Fault-Prone Multigrid Method is motivated by the classical analysis of the fault-free algorithm.
The standard assumptions for the convergence analysis of the fault-free multigrid method read as follows~\cite{Braess2007_FiniteElements,Hackbusch1985_MultiGridMethodsApplications,Hackbusch1994_IterativeSolutionLargeSparseSystemsEquations,TrottenbergOosterleeEtAl2001_Multigrid}:

\begin{assumption}[Smoothing property]\label{as:smoothing}
  There exists \(\eta: \NN\rightarrow \RR_{\geq0}\) satisfying \(\lim_{\nu\rightarrow\infty}\eta(\nu)  = 0\) and such that
  \begin{align*}
    \norm{\mat{A}_{\ell}\left(\mat{E}^{S}_{\ell}\right)^{\nu}}_{2} & \leq \eta(\nu)\norm{\mat{A}_{\ell}}_{2}, \quad \nu\geq0,~\ell=1,\dots,L.
  \end{align*}
\end{assumption}

\begin{assumption}[Approximation property]\label{as:approximation}
  There exists a constant \(C_{A}\) such that
  \begin{align*}
    \norm{\mat{E}^{CG}_{\ell}\mat{A}_{\ell}^{-1}}_{2} & \leq \frac{C_{A}}{\norm{\mat{A}_{\ell}}_{2}}, \quad \ell=1,\dots,L.
  \end{align*}
\end{assumption}
The smoothing and approximation property imply two grid convergence, with convergence rate given by
\begin{align*}
  \rho\left(\mat{E}^{TG}_{\ell}\right) &\leq \norm{\mat{E}^{CG}_{\ell}\left(\mat{E}^{S}_{\ell}\right)^{\nu_{1}+\nu_{2}}}_{2} \numberthis \label{eq:tgconv}\\
  &\leq \norm{\mat{E}^{CG}_{\ell}\mat{A}_{\ell}^{-1}}_{2}\norm{\mat{A}_{\ell}\left(\mat{E}^{S}_{\ell}\right)^{\nu_{1}+\nu_{2}}}_{2} \\
  &\leq C_{A}\eta\left(\nu_{1}+\nu_{2}\right).
\end{align*}
The right-hand side is less than one provided \(\nu_{1}+\nu_{2}\) is large enough.

\begin{assumption}\label{as:smoother}
  The smoothener is non-expansive, i.e.\ \(\rho\left(\mat{E}^{S}_{\ell}\right)=\norm{\mat{E}^{S}_{\ell}}_{A}\leq 1\), and there exists \(C_{S}>0\) such that
  \begin{align*}
    \norm{\left(\mat{E}^{S}_{\ell}\right)^{\nu}}_{2} & \leq C_{S} \quad \nu\geq 1,~\ell=1,\dots,L.
  \end{align*}
\end{assumption}
\Cref{as:smoother} permits to show that the two grid method is also a contraction with respect to the spectral norm \(\norm{\bullet}_{2}\):
\begin{align*}
  \norm{\mat{E}^{TG}_{\ell}}_{2} &\leq \norm{\left(\mat{E}^{S}_{\ell}\right)^{\nu_{2}}}_{2} \norm{\mat{E}^{CG}_{\ell}\left(\mat{E}^{S}_{\ell}\right)^{\nu_{1}}}_{2}
  \leq C_{S} C_{A}\eta\left(\nu_{1}\right).
\end{align*}
While this result is weaker than \cref{eq:tgconv}, it will be useful in the fault-prone case.

\begin{assumption}\label{as:prolongation}
  There exist positive constants \(\underline{C}_{p}\) and \(\overline{C}_{p}\) such that
  \begin{align*}
    \underline{C}_{p}^{-1}\norm{x_{\ell}}_{2} & \leq \norm{\mat{P} x_{\ell}}_{2}\leq \overline{C}_{p}\norm{x_{\ell}}_{2} \quad \forall x_{\ell}\in \RR^{n_{\ell}},~\ell=0,\dots,L-1.
  \end{align*}
\end{assumption}
\Cref{as:smoother,as:prolongation} allow to extend the convergence theory to the multilevel case with \(\gamma\geq2\).
The most interesting case in practice is the W-cycle (\(\gamma=2\)).
\begin{align*}
  \norm{\mat{E}_{\ell}}_{2} &= \norm{\left(\mat{E}^{S}_{\ell}\right)^{\nu_{2}} \left[\mat{I}-\mat{P}\left(\mat{I}-\mat{E}_{\ell-1}^{\gamma}\right)\mat{A}_{\ell-1}^{-1}\mat{R}\mat{A}_{\ell}\right]  \left(\mat{E}^{S}_{\ell}\right)^{\nu_{1}}}_{2} \\
  &\leq \norm{\left(\mat{E}^{S}_{\ell}\right)^{\nu_{2}} \left[\mat{I}-\mat{P}\mat{A}_{\ell-1}^{-1}\mat{R}\mat{A}_{\ell}\right]  \left(\mat{E}^{S}_{\ell}\right)^{\nu_{1}}}_{2} \\
  &\quad + \norm{\left(\mat{E}^{S}_{\ell}\right)^{\nu_{2}}}_{2} \norm{\mat{P}}_{2} \norm{\mat{E}_{\ell-1}}_{2}^{\gamma} \norm{\mat{A}_{\ell-1}^{-1}\mat{R}\mat{A}_{\ell} \left(\mat{E}^{S}_{\ell}\right)^{\nu_{1}}}_{2} \\
  &\leq \norm{\mat{E}^{TG}_{\ell}}_{2} + \underline{C}_{p}\overline{C}_{p}C_{S} \norm{\mat{E}_{\ell-1}}_{2}^{\gamma} \norm{\mat{P}\mat{A}_{\ell-1}^{-1}\mat{R}\mat{A}_{\ell} \left(\mat{E}^{S}_{\ell}\right)^{\nu_{1}}}_{2}.
\end{align*}
Since
\begin{align*}
  \norm{\mat{P}\mat{A}_{\ell-1}^{-1}\mat{R}\mat{A}_{\ell} \left(\mat{E}^{S}_{\ell}\right)^{\nu_{1}}}_{2}
  &\leq\norm{\left(\mat{I}-\mat{P}\mat{A}_{\ell-1}^{-1}\mat{R}\mat{A}_{\ell}\right) \left(\mat{E}^{S}_{\ell}\right)^{\nu_{1}}}_{2} + \norm{\left(\mat{E}^{S}_{\ell}\right)^{\nu_{1}}}_{2} \\
  &\leq C_{S} + C_{A}\eta\left(\nu_{1}\right),
\end{align*}
we obtain the recursive inequality
\begin{align}
  \norm{\mat{E}_{\ell}}_{2} &\leq \norm{\mat{E}^{TG}_{\ell}}_{2} + C_{*}\norm{\mat{E}_{\ell-1}}_{2}^{\gamma} \label{eq:18}
\end{align}
with \(C_{*}=C_{S}\underline{C}_{p}\overline{C}_{p}\left(C_{S}+C_{A}\eta\left(\nu_{1}\right)\right)\).
A classical result \cite{Hackbusch1985_MultiGridMethodsApplications} concerning this inequality is
\begin{lemma} \label{lem:reccursion}
  Suppose the elements of the sequence \(\left\{\eta_{\ell}\right\}_{\ell\geq1}\) satisfy \(0\leq \eta_{1} \leq \xi\) and \(0\leq \eta_{\ell} \leq \xi + C_{*}\eta_{\ell-1}^{\gamma}\), \(\ell\geq 2\).
  If \(\gamma\geq 2\), \(C_{*}\gamma > 1\) and \(\xi\leq \frac{\gamma-1}{\gamma}\frac{1}{\left(C_{*}\gamma\right)^{\frac{1}{\gamma-1}}}\), then
  \begin{align*}
    \eta_{\ell}&\leq \left\{
                 \begin{array}{ll}
                   \frac{\gamma}{\gamma-1}\xi, & \gamma\geq 2, \\
                   \frac{2\xi}{1+\sqrt{1-4C_{*}\xi}}, & \gamma=2
                 \end{array}\right\} < 1.
  \end{align*}
\end{lemma}
The result show that two grid convergence along with a sufficient number of smoothing steps imply the multilevel scheme is convergent in the absence of faults.

\subsection{Review of Previous Work on the Fault-Prone Two Grid Method}

The iteration matrix of the Fault-Prone Multigrid Method is random, and the usual the spectral radius used in the fault-free case is no longer relevant.
Instead, it transpires that the asymptotic rate of convergence in the fault-prone case is governed by the \emph{Lyapunov spectral radius} of the iteration matrix:
\begin{align*}
  \varrho(\randmat{E}_L)=\lim_{N\rightarrow\infty}\exp\left\{\Exp{\log\norm{\randmat{E}_{L}^{N}}^{1/N}}\right\}.
\end{align*}
We refer the reader to \cite{BougerolLacroix1985_ProductsRandomMatricesWith,CrisantiPaladinEtAl1993_ProductsRandomMatrices,AinsworthGlusa216_IsMultigridMethodFaultTolerant} for further discussion and details relating to the Lyapunov spectral radius.
In particular, \cite{AinsworthGlusa216_IsMultigridMethodFaultTolerant} describes the so called \emph{Replica trick} which gives the following bound for the Lyapunov spectral radius
\begin{align}
  \varrho\left(\randmat{E}_{L}\right) & \leq \sqrt{\rho\left(\Exp{\randmat{E}_{L}\otimes\randmat{E}_{L}}\right)}. \label{eq:replicaTrick}
\end{align}
Under suitable assumptions, it was shown \cite{AinsworthGlusa216_IsMultigridMethodFaultTolerant} that the Lyapunov spectral radius for the Fault-Prone Two Grid Method satisfies
\begin{align}
  \varrho \left(\randmat{E}^{TG}_{L}\right) & \leq \norm{\mat{E}^{TG}_{L}}_{A} +
                                       C\begin{cases}
                                         \varepsilon n_{L}^{(4-d)/2d} & d < 4, \\
                                         \varepsilon \sqrt{\log n_{L}} & d=4, \\
                                         \varepsilon & d > 4,
                                       \end{cases} \label{eq:TGbound}
\end{align}
where \(d\) is the spatial dimension of the underlying PDE, and \(\norm{\bullet}_{A}\) is the usual energy norm.
The estimate suggests, as confirmed by numerical examples, that the convergence rate of the Fault-Prone Two Grid Method degenerates as \(n_{L}\rightarrow\infty\) and eventually fails to converge.
Moreover, it can be shown \cite{AinsworthGlusa216_IsMultigridMethodFaultTolerant} that protection of the prolongation operation against faults (so that \(\randmat{X}^{(P)}_{L}=\mat{I}\)), whilst allowing other sources of faults to remain, results in the Two Grid scheme being resilient:
\begin{theorem}\label{thm:TGNoProlong}
  Let \(\randmat{E}^{TG}_{\ell}\left(\nu_{1},\nu_{2}\right) = \left(\randmat{E}^{S,\text{post}}_{\ell}\right)^{\nu_{2}} \randmat{E}^{CG}_{\ell} \left(\randmat{E}^{S,\text{pre}}_{\ell}\right)^{\nu_{1}}\)
  be the iteration matrix of the Fault-Prone Two Grid Method with faults in smoothener, residual and restriction, and protected prolongation.
  Provided \Cref{as:faults,as:smoothing,as:approximation,as:smoother,as:prolongation} hold, and that \(\mat{N}_{\ell}\) and \(\randmat{X}^{(S)}_{\ell}\) commute, we find that
  \begin{align*}
    \varrho \left(\randmat{E}^{TG}_{\ell}\left(\nu_{1},\nu_{2}\right)\right) & \leq \opnorm{\Exp{\left(\randmat{E}^{TG}_{\ell}\left(\nu_{1},\nu_{2}\right)\right)^{\otimes2}}}^{1/2} \leq \norm{\mat{E}^{TG}_{\ell}\left(\nu_{2},\nu_{1}\right)}_{2} + C\varepsilon,
  \end{align*}
  where the constant \(C\) is independent of \(\varepsilon\) and \(\ell\), and \(\opnorm{\bullet}\) is the double energy norm defined in the Appendix.
\end{theorem}
\Cref{thm:TGNoProlong} shows that the convergence of the Fault-Prone Two Grid scheme with a protected prolongation does not degenerate as \(n_{L}\rightarrow\infty\).
One of the main aims of the present work is to generalise \Cref{thm:TGNoProlong} to the case of the Fault-Prone Multigrid Method.

\section{Main Results}
\label{sec:main-results}

Consider a second order elliptic PDE with homogeneous Dirichlet boundary conditions on a polyhedral domain \(\Omega\) given in variational form by
\begin{align*}
  \text{Find } u\in V = H_{0}^{1}\left(\Omega\right): a(u,v) = L(v) \quad \forall v\in V.
\end{align*}
Here, the bilinear form \(a\) is continuous and \(V\)-coercive, and the linear form \(L\) is continuous.
Using a shape regular partitioning \(\c{T}_{L}\) of \(\Omega\), the finite dimensional space of continuous piecewise polynomials of order \(k\) is defined as
\begin{align*}
  V_{L}=\left\{v\in H_{0}^{1}\left(\Omega\right) \mid v|_{K} \in \PP_{k}(K)~\forall K \in \c{T}_{L} \right\}.
\end{align*}
Letting \(\phi_{L}\) be the vector of nodal basis functions \(\phi_{L}^{(i)}\), \(i=1,\dots,n_{L}\), the solution \(u_{L}\in V_{L}\) to the PDE is approximated by
\begin{align*}
  \mat{A}_{L}x_{L}=b_{L},
\end{align*}
where \(\mat{A}_{L}=a\left(\phi_{L},\phi_{L}\right)\), \(b_{L}=L\left(\phi_{L}\right)\), and \(u_{L}=\phi_{L}\cdot x_{L}\).

The hierarchy of levels for the multigrid solver is constructed from discretizations of the same problem on nested coarser meshes \(\c{T}_{\ell}\).
Uniform refinement \cite{KroegerPreusser2008_Stability8TetrahedraShortest} or adaptive mesh refinement can be used to obtain successively finer meshes, and the canonical injection of the coarser space into the finer one is used to define prolongation and restriction.

\subsection{Convergence Estimate for the Fault-Prone W-Cycle}

While the two grid method is, as a solver,  mostly of academic interest, the behaviour of the multigrid method under the impact of faults is of great practical importance.
We extend the result of \Cref{thm:TGNoProlong} to the W-cycle.
This theorem mirrors the classical implication of W-cycle convergence by two grid convergence, but applies to the Lyapunov spectral radius needed for the analysis of the Fault-Prone Multigrid Method.
No additional assumptions are required beyond these needed for the classical multigrid analysis.

\begin{restatable}{theorem}{MGNoProlong}
\label{thm:MGNoProlong}
  Provided \cref{as:smoothing,as:approximation,as:smoother,as:prolongation,as:faults} hold, that the prolongation is protected, that the number of smoothing steps is sufficient, and that \(\varepsilon\) is sufficiently small, the fault-prone multigrid method converges with a rate bounded by
  \begin{align*}
    \varrho\left(\randmat{E}_{\ell}\left(\nu_{1},\nu_{2},\gamma\right)\right)
    & \leq
      \begin{cases}
        \frac{\gamma}{\gamma-1}C_{TG} + C\varepsilon, & \gamma\geq 2,\\
        \frac{2}{1 + \sqrt{1-4C_{*}C_{TG}}}C_{TG} + C\varepsilon, & \gamma=2,
      \end{cases}
  \end{align*}
  where \(C\) is independent of \(\varepsilon\) and \(\ell\) and
  \begin{align*}
    C_{TG} &= \max_{\ell}\norm{\mat{E}_{\ell}^{TG}\left(\nu_{2},\nu_{1}\right)}_{2} \leq C_{S}C_{A}\eta(\nu_{2}), \\
    C_{*}  &= C_{S}\underline{C}_{p}\overline{C}_{p}\left(C_{S}+C_{A}\eta(\nu_{2})\right).
  \end{align*}
\end{restatable}

Setting the fault rate \(\varepsilon=0\) in the above bounds recovers the classical estimates from \Cref{sec:review} for the fault-free multigrid method, since \(\varrho\left(\randmat{E}_{\ell}\right)\) reduces to \(\rho\left(\mat{E}_{\ell}\right)\) for \(\varepsilon=0\).
Just as the Two Grid result, this Theorem makes no assumptions about the origin of the solver hierarchy, and does not rely on a particular choice of smoothener.
We refer the reader to \Cref{sec:proof} for the proof of \Cref{thm:MGNoProlong}.

\subsection{The Effect of Protection of the Prolongation}

\Cref{thm:MGNoProlong} assumes that the prolongation operator is protected.
We investigate whether this assumption can be relaxed by considering a numerical example where we do not protect the prolongation.
Specifically, we consider the Poisson equation in two dimensions
\begin{align}
  \left\{
  \begin{array}{rll}
    -\Delta u &= f & \text{in }\Omega=[0,1]^{2}, \\
    u&=0&\text{on }\partial\Omega,
  \end{array}\right. \label{eq:poisson2d}
\end{align}
and in order to rule out extraneous effects, we use a uniform mesh, a discretization with piecewise linear finite elements, and optimally damped Jacobi pre- and post-smootheners.
In particular, it is well established that the fault-free multigrid method converges in this setting.
\Cref{as:approximation,as:smoother,as:smoothing,as:prolongation} are satisfied, as for example shown in \cite{Hackbusch1994_IterativeSolutionLargeSparseSystemsEquations}.
On every level, residual, restriction, prolongation and smootheners are subject to componentwise faults, as given in \cref{eq:componentwiseFaults}.
We consider problems of size between 1 million and 1 billion degrees of freedom, and fault probabilities between \(10^{-4}\) and \(0.6\).
To minimize floating point contamination in the approximation of the Lyapunov spectral radius, the right-hand side \(f\) is taken to be zero, a non-zero random initial iterate is chosen, and after each iteration the current iterate is renormalised.
In \Cref{fig:poisson2dgeoWcycle-residuals}, we plot the evolution of the residual norm with respect to the iteration number for the case of laissez-faire mitigation in prolongation, restriction, residual and smoothener.
We can see that the number of degrees of freedom \(n_{L}\) adversely affects the rate of convergence, even leading to divergence for large number of unknowns.
Estimates of the Lyapunov spectral radius are obtained as the geometric average of 1000 iterations of Fault-Prone Multigrid, and are displayed in \Cref{fig:poisson2dgeoWcycle-qd-rhoL}.
\begin{figure}
  \centering
  \includegraphics{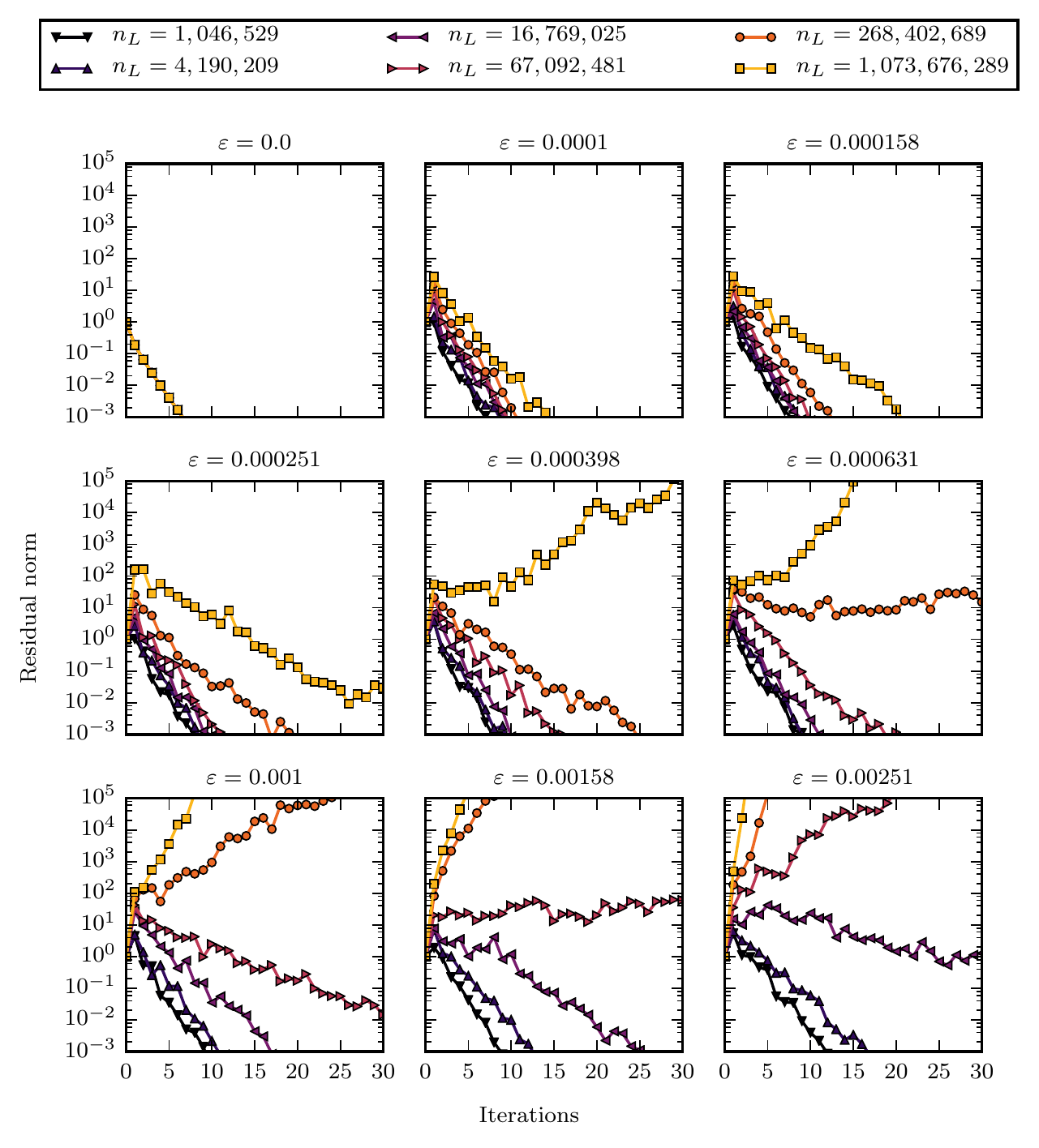}
  \caption{
    Plots of the norm of the residual against iteration number for the Fault-Prone W-Cycle Multigrid Method in the case of discretization of Poisson problem on square domain.
  } \label{fig:poisson2dgeoWcycle-residuals}
\end{figure}



\begin{figure}
  \centering
  \label{fig:poisson2dgeoWcycle-qd-rhoL}
  \label{fig:poisson2dgeoWcycleNoProlong-qd-rhoL}
  \includegraphics{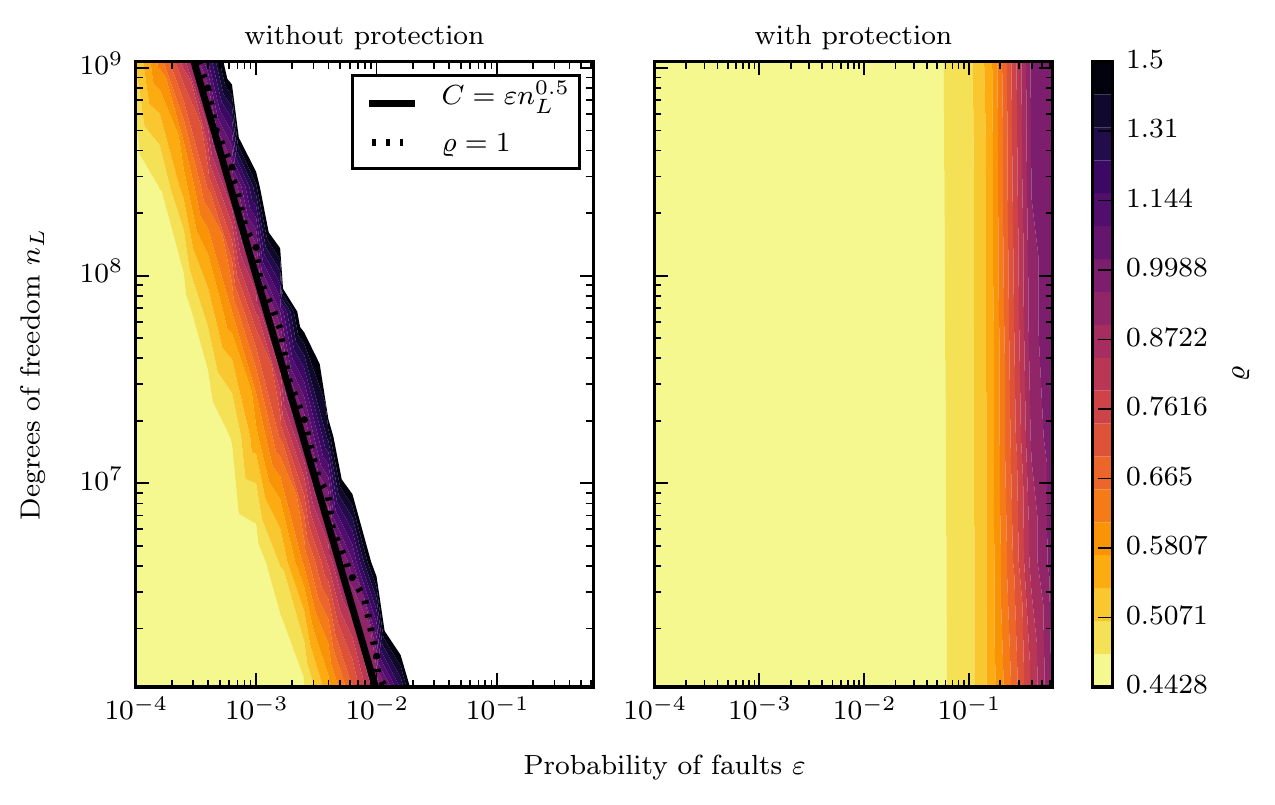}
  \caption{
    Lyapunov spectral radius $\varrho(\randmat{E}_L)$ for the iteration matrix for the Fault-Prone W-Cycle Multigrid method in the case of discretization of Poisson problem on a square domain.
    Without protected prolongation (\textit{left}), and with protected prolongation (\textit{right}).
  }
\end{figure}

In the Two Grid setting which was discussed in \cite{AinsworthGlusa216_IsMultigridMethodFaultTolerant}, we found that the asymptotic rate for this problem scales like \(\sqrt{n_{L} \varepsilon^{2}}\).
Similarly, as seen in \Cref{fig:poisson2dgeoWcycle-qd-rhoL}, we find
\begin{align}
  \varrho\left(\randmat{E}_{L}\right) = C_{0} + \c{O}\left(\sqrt{n_{L} \varepsilon^{2}}\right) \label{eq:5}
\end{align}
for the multilevel case, with \(C_{0}\) being a constant related to the fault-free method.
This means that the method is not fault resilient, since for any given rate of faults \(\varepsilon\), there exists a maximum problem size above which multigrid diverges.
Therefore additional protection is mandatory in the multilevel case as well.

We repeat the same experiments, this time with a protected prolongation, i.e. \(\randmat{X}_{\ell}^{(P)}=\mat{I}\).
This produces the desired independence of the problem size, as predicted by \Cref{thm:MGNoProlong} and as shown by the evolution of the residual in \Cref{fig:poisson2dgeoWcycleNoProlong-residuals} and by the estimated rate of convergence in \Cref{fig:poisson2dgeoWcycleNoProlong-qd-rhoL}.
We further illustrate the results with several test problems that pose more of a challenge to the multigrid solver.

\begin{figure}
  \centering
  \includegraphics{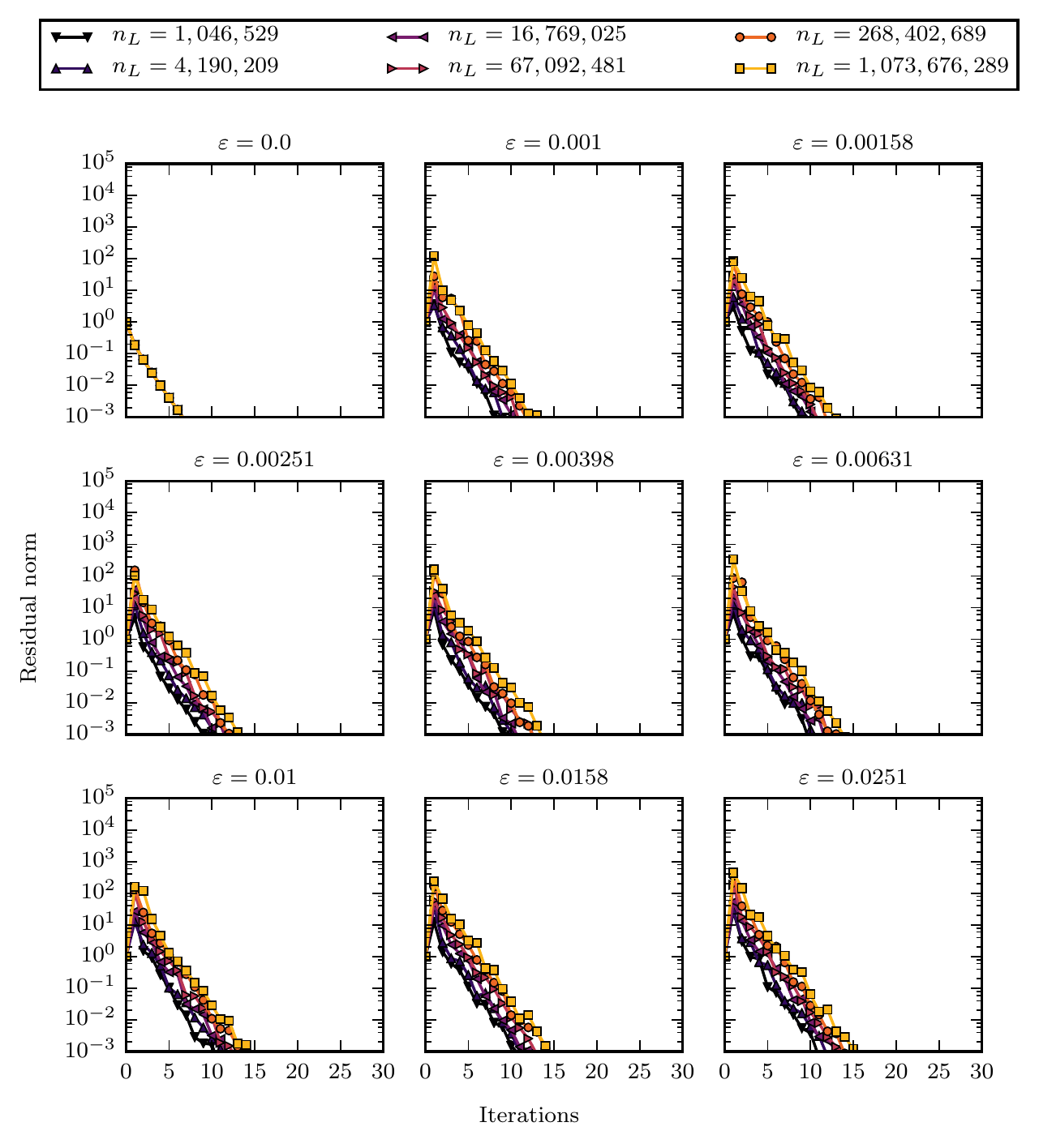}
  \caption{
    Plot of the norm of the residual against iteration number for the Fault-Prone W-Cycle Multigrid Method with protected prolongation in the case of discretization of Poisson problem on square domain.
  } \label{fig:poisson2dgeoWcycleNoProlong-residuals}
\end{figure}

\subsection{Adaptively Refined Meshes}

A second numerical example illustrates the results of \Cref{thm:MGNoProlong} for the case of an adaptively refined mesh.
We consider the 2D magnetostatics problem for a three phase 6/4 switched reluctance motor as depicted in \Cref{fig:motorMesh}.
Gauss's law and Amp\`ere's law are given by
\begin{align*}
  \Div \vec{B} &= 0, &
  \curl \vec{H} &= J,
\end{align*}
where \(\vec{B}\) is the magnetic flux density, \(\vec{H}\) the magnetic field intensity and \(J\) the current density.
\(\vec{B}\) and \(\vec{H}\) are linked through the constitutive relation \(\vec{B} = \mu \vec{H}\) with magnetic permeability \(\mu\).
Using the magnetic vector potential \(\vec{B}=\vec{\curl}~ u\), the system can be rewritten as
\begin{align*}
  -\frac{1}{\mu}\Delta u &= J.
\end{align*}
This gives rise to a variational problem with
\begin{align*}
  a(u,v)&= \int_{\Omega} \frac{1}{\mu}\grad u \cdot \grad v, &
  F(v) &= \int_{\Omega} J v.
\end{align*}
The permeability \(\mu\) is \(5200\) in the rotor and the stator, and \(1\) everywhere else.
The current density \(J\) is unity in the coils and \(0\) everywhere else.
Using continuous piecewise linear finite elements and classical residual-based local a posteriori error indicators \cite{ErnGuermond2004_TheoryPracticeFiniteElements}, we adaptively refine an initial mesh shown in \Cref{fig:motorMesh}.
We use Red-Green refinement \cite{BankShermanEtAl1983_RefinementAlgorithmsDataStructures} coupled with a D\"orfler marking strategy \cite{Doerfler1996_ConvergentAdaptiveAlgorithmPoissonsEquation} with refinement parameter 0.5.
The solver hierarchy is generated without injection of faults; then the problem is solved for componentwise faults and varying fault rates.

\Cref{fig:motor-qd-rhoL} depicts the approximate rate of convergence of the W-Cycle with one step of Jacobi pre- and post-smoothing each and without protection of the prolongation.
We can see that although neither the W-Cycle nor the adaptively refined mesh are covered by the two grid result from \cite{AinsworthGlusa216_IsMultigridMethodFaultTolerant}, the convergence estimate is qualitatively correct.
Added protection of the prolongation operation leads to a fault resilient method, as shown in \Cref{fig:motorNoProlong-qd-rhoL}.
Small variations with respect to the number of degrees of freedom are due to varying coarsening ratios of the levels.
The results match \Cref{thm:MGNoProlong}, even though \Cref{as:approximation} is not satisfied.

\begin{figure}
  \centering
  \includegraphics{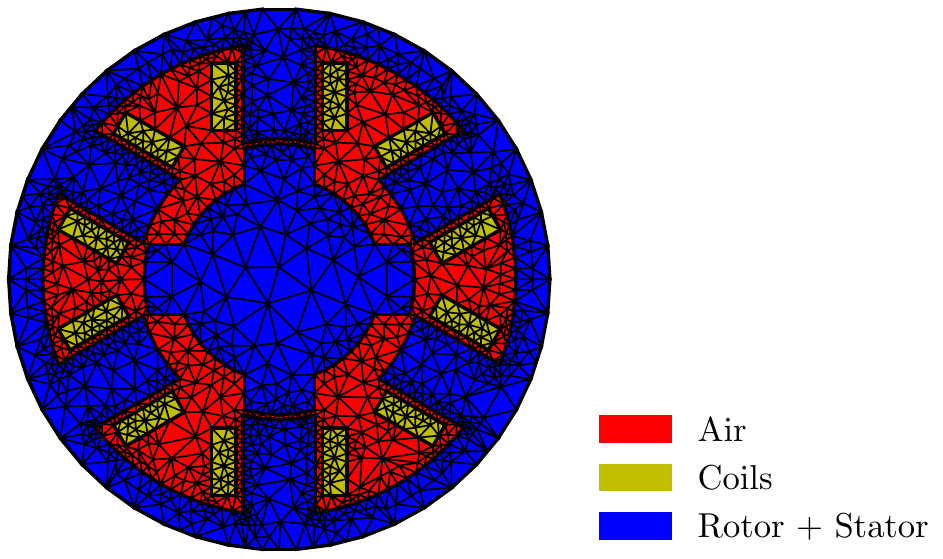}
  \caption{
    Initial mesh for the motor problem.
  } \label{fig:motorMesh}
\end{figure}



\begin{figure}
  \centering
  \label{fig:motor-qd-rhoL}
  \label{fig:motorNoProlong-qd-rhoL}
  \includegraphics{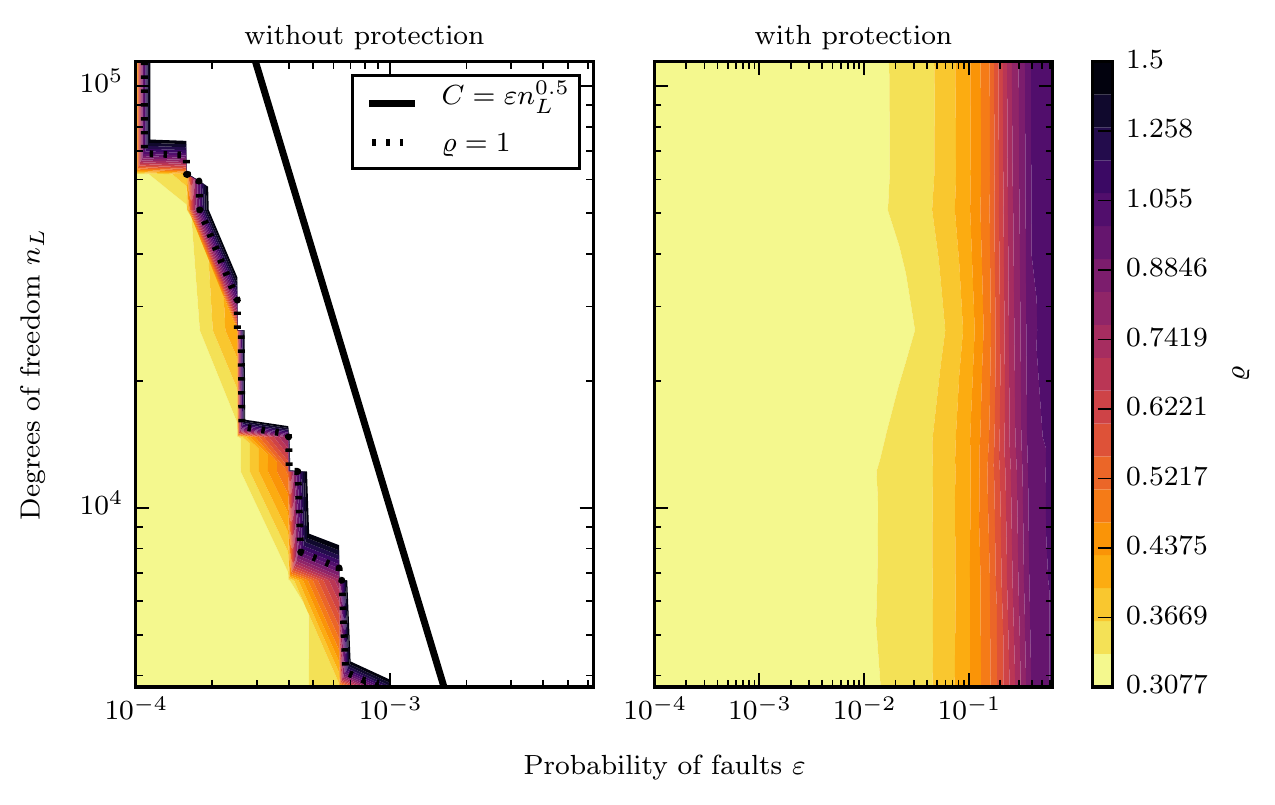}
  \caption{
    Lyapunov spectral radius $\varrho(\randmat{E}_L)$ for the iteration matrix for the Fault-Prone W-Cycle Multigrid Method in the case of discretization of the motor problem.
    Without protected prolongation (\textit{left}), and with protected prolongation (\textit{right}).
  }
\end{figure}

\subsection{Higher Spatial Dimension and Higher Order Finite Elements}

We now demonstrate that fault resilience of the W-cycle is retained for a 3D partial differential equation and higher order (quadratic) continuous finite elements, and solve the Poisson equation on the Fichera cube.
\begin{align*}
  -\Delta u &= f && \text{in } \Omega=[0,2]^{3}\setminus [1,2]^{3}, \\
  u&=u_{0} && \text{on } \partial\Omega
\end{align*}
The geometry and its uniform meshing are shown in \Cref{fig:fichera}.
The system is partitioned using METIS \cite{KarypisKumar1998_FastHighQualityMultilevel} into blocks of size \(2^{14}\) and distributed over the compute nodes.
We inject nodewise faults as given by \cref{eq:blockwiseFaults}.
The problem sizes considered range from 1.7 million to about 940 million degrees of freedom.
In \Cref{fig:fichera-qd-rhoL}, we show the approximate rate of convergence as well as a level set of the error bound obtained for the Fault-Prone Two Grid Method in \cite{AinsworthGlusa216_IsMultigridMethodFaultTolerant}.
It can be seen that the rate of convergence in the multilevel case behaves like
\begin{align*}
  \varrho\left(\randmat{E}_{L}\right) = C_{0} + \c{\varepsilon\sqrt[6]{n_{L}}},
\end{align*}
where \(C_{0}\) is a constant that is related to the fault-free method.
We notice that the onset of divergent behaviour indeed happens at a slower rate as compared to the 2D setting.
Protection of the prolongation again makes the method fault resilient, as shown in \Cref{fig:ficheraNoProlong-qd-rhoL}.
The choice of higher order elements plays no significant role in the convergence behaviour of the method.

\begin{figure}
  \centering
  \includegraphics[scale=0.4]{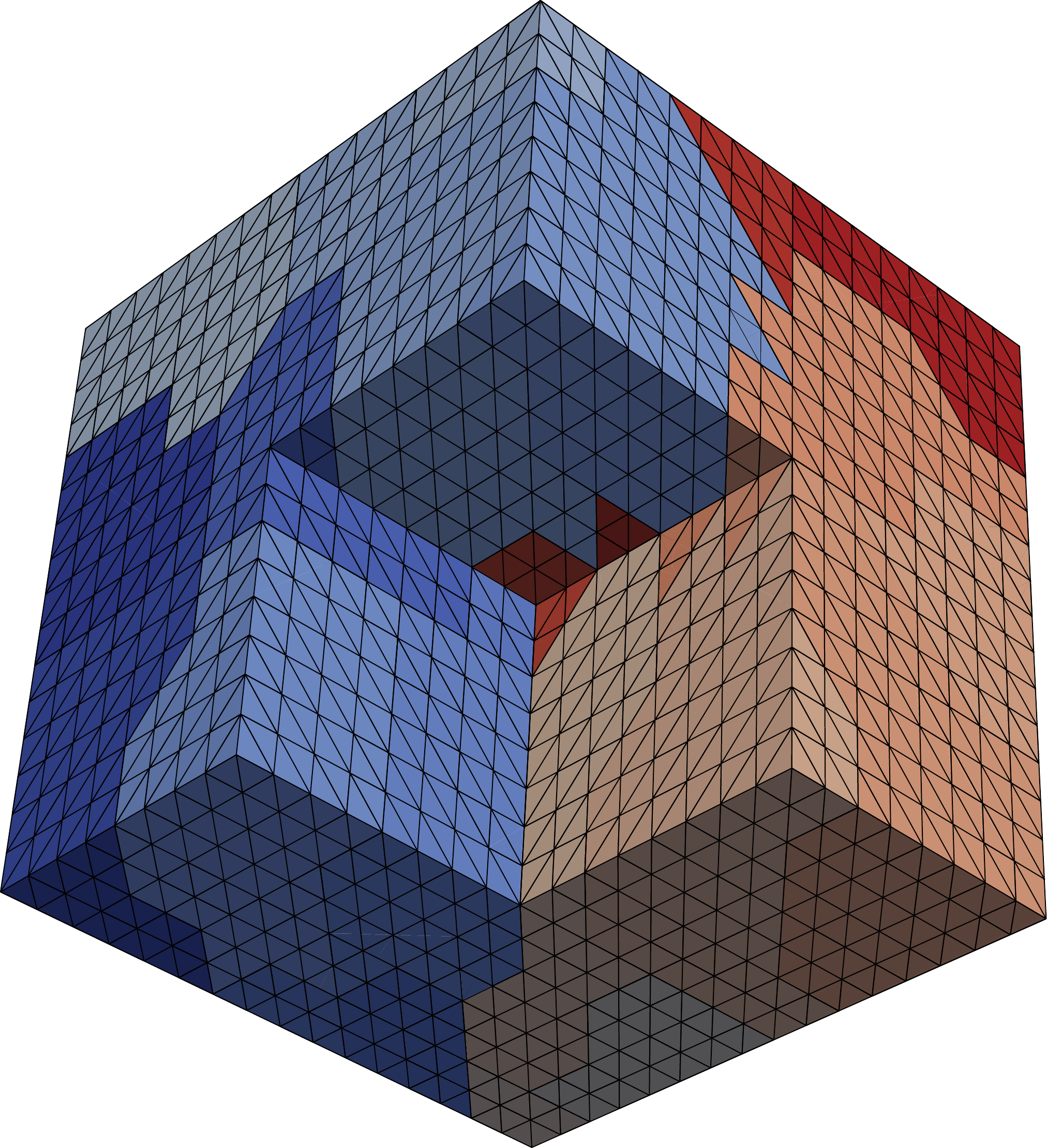}
  \caption{
    Mesh for the Fichera cube, indicating partitioning from METIS.
  } \label{fig:fichera}
\end{figure}



\begin{figure}
  \centering
  \label{fig:fichera-qd-rhoL}
  \label{fig:ficheraNoProlong-qd-rhoL}
  \includegraphics{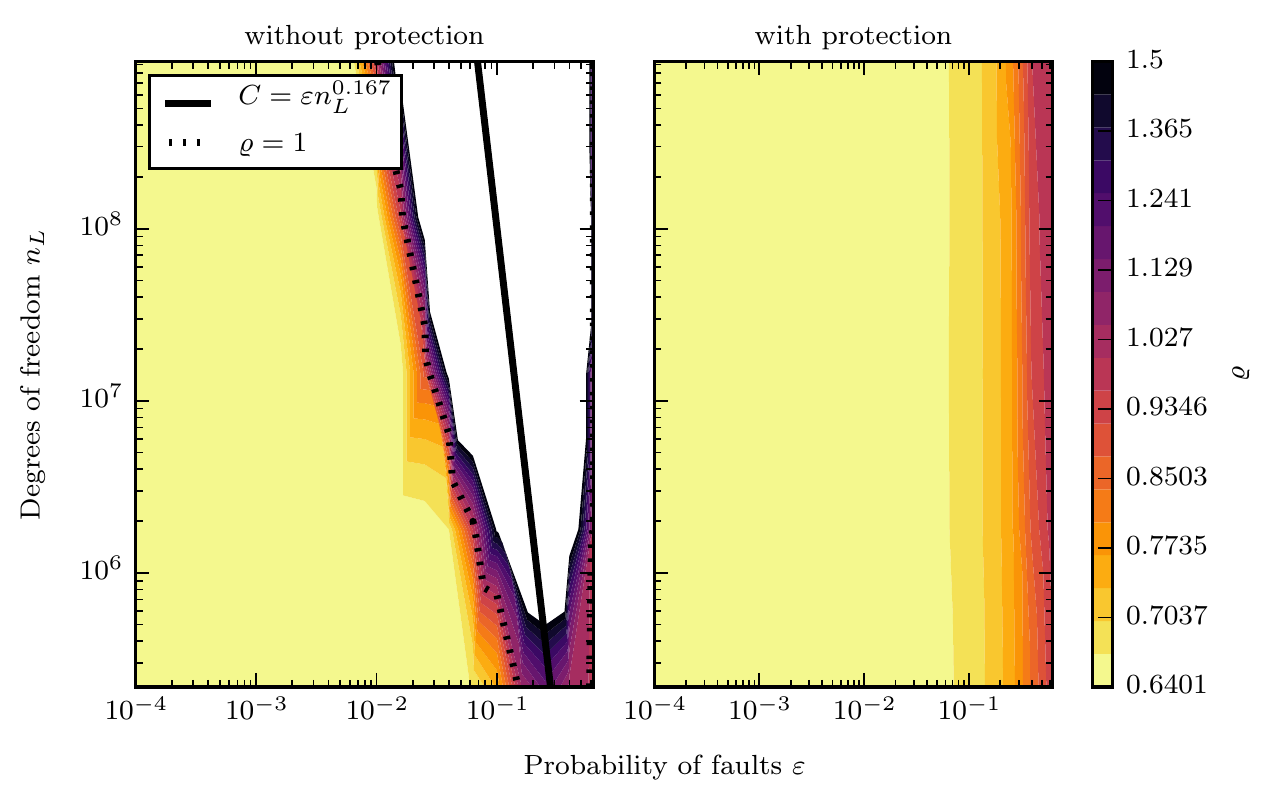}
  \caption{
    Lyapunov spectral radius $\varrho(\randmat{E}_L)$ for the iteration matrix for the Fault-Prone W-Cycle Multigrid Method in the case of discretization of the Fichera cube.
    Without protected prolongation (\textit{left}), and with protected prolongation (\textit{right}).
  }
\end{figure}

\FloatBarrier

\section{Implementation Issues}
\label{sec:implementation}

In the previous examples, we assumed that faults can be perfectly detected.
In practice, faults cannot be perfectly detected, nor is is it possible to perfectly protect the prolongation.
The question therefore arises of whether \Cref{thm:MGNoProlong} has any practical relevance.
In \Cref{sec:detect-soft-faults}, we present one possible simple approach for fault detection and show that the behaviour of Fault-Prone Multigrid found in \cref{eq:5} is recovered.
In \Cref{sec:prot-prol} we turn to the issue of protecting the prolongation operator.

\subsection{Detection of Soft Faults}
\label{sec:detect-soft-faults}

The laissez-faire fault mitigation strategy requires fault detection.
While this is straight-forward in the case of hard faults, soft faults are more problematic.
Various techniques have been suggested \cite{SnirWisniewskiEtAl2014_AddressingFailuresExascaleComputing}.
Here, we present a simple approach based on replication \cite{HeraultRobert2015_FaultToleranceTechniquesHighPerformanceComputing} in which a fault-prone component is repeated \(K\) times (\(K\geq2\)) and the results are compared for consistency.
The replication of node local operations is free of any communication requirements.
\Cref{alg:detection} shows how the strategy is used in conjunction with laissez-faire in the detection of faults in the computation of a generic matrix-vector product \(y=\mat{M}x\).
Instead of the action of \(\mat{M}\), only its fault-prone equivalent \(\randmat{M}\) is available in practice, where \(\randmat{M}\in\RR^{n\times n}\) is a random matrix.

\begin{algorithm}
  \begin{algorithmic}[1]
    \For{$i\leftarrow 1$ \To{} $n$}
    \For{$j\leftarrow 1$ \To{} $K$}
    \State \(w_{j} \leftarrow (\randmat{M}x)_{i}\) \Comment{\(j\)-th replica}
    \EndFor
    \If{\(w_{1}=w_{2}=\dots=w_{K}\) and \(\abs{w_{1}} < 10^{16}\)}
    \State \(y_{i} \leftarrow w_{1}\) \Comment{value accepted}
    \Else
    \State \(y_{i} \leftarrow 0\) \Comment{Laissez-faire}
    \EndIf
    \EndFor
  \end{algorithmic}
  \caption{Detection of faults using \(K\) replicas and laissez-faire mitigation for the fault-prone matrix-vector product \(y=\randmat{M}x\).}
  \label{alg:detection}
\end{algorithm}

The basic idea behind \Cref{alg:detection} is to declare an operation as being fault-free if all replicas are in agreement, otherwise a fault is deemed to have occurred and laissez-faire is triggered.
This means that there are three possible outcomes for the \(i\)-th component of the output \(y\) from \Cref{alg:detection}:
\begin{align*}
  &(C) && \text{ the correct value of \(y_{i}=(\mat{M}x)_{i}\) is returned,} \\
  &(M) && \text{ an error is detected, the laissez-faire mitigation is triggered, and \(y_{i}=0\)}, \\
  &(U) && \text{ a fault occurs but remains undetected, and \(y_{i}\) is a corrupted value.}
\end{align*}
Our analysis in \Cref{sec:main-results} caters for the cases (\(C\)) and (\(M\)), but does not take account of case (\(U\)).
Suppose that the probability of a replica being corrupted is \(\varepsilon\ll 1\).
Then the probability of \emph{all} \(K\) replicas being corrupted in \emph{exactly} the same way is \(\c{O}(\varepsilon^{K})\), showing that the likelihood of (\(U\)) occurring is exponentially small in \(K\).
Hence, the probabilities of the outcomes of \Cref{alg:detection} are
\begin{align*}
  \Pr{C} &=1-\c{O}(\varepsilon), \\
  \Pr{M} &=\c{O}(\varepsilon), \\
  \Pr{U} &=\c{O}(\varepsilon^{K}).
\end{align*}

We illustrate the effect of using our ad-hoc fault detection strategy given in \Cref{alg:detection} in \Cref{fig:poisson2dgeoSoftWcycle2-qd-rhoL} for the Poisson problem \cref{eq:poisson2d}.
In particular, the faults are modelled using bit flipping in which any bit in the floating point representation is flipped with a small but non-zero probability such that the overall probability of a floating point number being corrupted is \(\varepsilon\).
The results are given in \Cref{fig:poisson2dgeoSoftWcycle2-qd-rhoL} in the simplest case where \(K=2\) replicas are used in \Cref{alg:detection} and all operations.
It is observed that, even in this simplest case, the convergence behaviour mirrors that which would be obtained with perfect detection.

\begin{figure}
  \centering
  \includegraphics{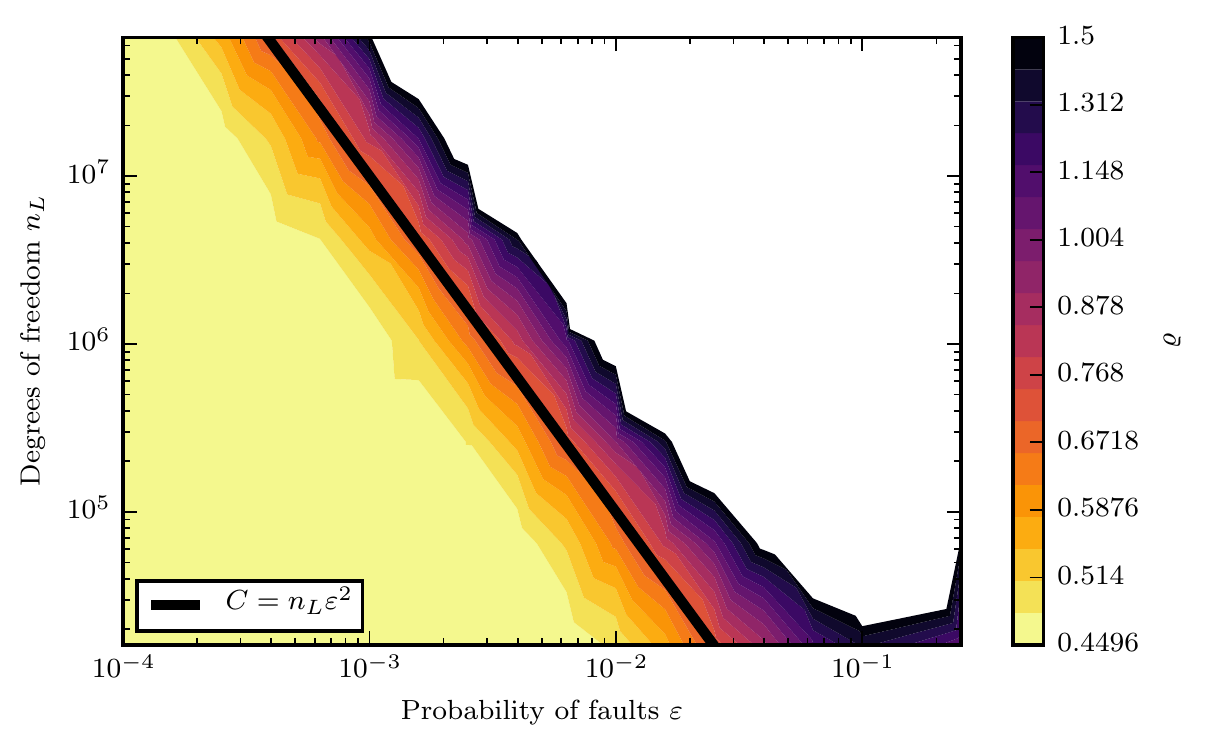}
  \caption{
    Lyapunov spectral radius $\varrho(\randmat{E}_L)$ for the iteration matrix for the Fault-Prone W-Cycle Multigrid Method in the case of discretization of the 2D Poisson problem with fault detection using \(K=2\) replicas in every operation.
  } \label{fig:poisson2dgeoSoftWcycle2-qd-rhoL}
\end{figure}

\subsection{Protection of the Prolongation}
\label{sec:prot-prol}

\Cref{thm:MGNoProlong} requires that the prolongation operations
\begin{align*}
  x_{\ell}\leftarrow x_{\ell} + \mat{P}e_{\ell-1}
\end{align*}
are computed exactly, i.e. without \emph{any} faults.
This is clearly not possible in practice.
We have seen that in the absence of full protection of the prolongation, the Fault-Prone Multigrid converges at a rate given by \cref{eq:5}, where \(\varepsilon\) is the underlying failure rate of the machine.
Moreover, \Cref{thm:MGNoProlong} suggests that the \(\c{O}(\sqrt{n_{L}\varepsilon^{2}})\) term is entirely due to faults in the prolongation (with the faults coming purely from other components of multigrid contributing with higher order terms).
At any rate, it is clear that if we can enhance the reliability of the prolongation \emph{sufficiently} (without being necessarily exact), then one can expect to ameliorate the factor \(\c{O}\left(\sqrt{n_{L}\varepsilon^{2}}\right)\) in the bound of the Lyapunov spectral radius, e.g. by obtaining a growth of \(\c{O}\left(n_{L}\varepsilon^{\alpha}\right)\) for some \(\alpha>2\).

Is it possible to improve the likelihood of detecting and mitigating faults in the prolongation beyond \(\Pr{M}=\c{O}\left(\varepsilon\right)\)?
\Cref{alg:detection} may be regarded as being overly conservative.
For instance, if all but one of the replicas are in agreement, then intuitively it seems likely that the majority are correct.
\Cref{alg:protection} implements the approach suggested by this argument by looking for agreement amongst a subset of \(k_{P}\) of the \(K_{P}\) replicas for the prolongation.
\begin{algorithm}
  \begin{algorithmic}[1]
    \For{$i\leftarrow 1$ \To{} $n_{\ell}$}
    \For{$j\leftarrow 1$ \To{} $K_{P}$}
    \State \(w_{j} \leftarrow (\randmat{P}e_{\ell-1})_{i}\) \Comment{\(j\)-th replica}
    \If{\(k_{P}\) replicas have matching value \(w\) and \(\abs{w}<10^{16}\)}
    \State \((x_{\ell})_{i} \leftarrow (x_{\ell})_{i}  + w\) \Comment{value accepted}
    \State \Break
    \EndIf
    \EndFor
    \EndFor
  \end{algorithmic}
  \caption{Protection of the fault-prone prolongation \(x_{\ell}\leftarrow x_{\ell}+\randmat{P}e_{\ell-1}\) with up to \(K_{P}\) replicas, acceptance threshold \(k_{P}\) and laissez-faire mitigation.}
  \label{alg:protection}
\end{algorithm}
This added freedom alters the likelihood at which the three events occur in the prolongation:
\begin{align*}
  \Pr{C_{P}} & =1-\c{O}(\varepsilon^{K_{P}-k_{P}+1}), \\
  \Pr{M_{P}} & =\c{O}(\varepsilon^{K_{P}-k_{P}+1}), \\
  \Pr{U_{P}} & =\c{O}(\varepsilon^{k_{P}}).
\end{align*}
 In particular, when \Cref{alg:protection} is applied to the computation of the prolongation with parameters \(k_{P}=3\) and \(K_{P}=4\), we obtain an overall rate of \(\Pr{M_{P}}=\c{O}(\varepsilon^{2})\) at which mitigation occurs.
 \Cref{fig:poisson2dgeoSoftWcycle4-qd-rhoL} shows the results obtained by applying \Cref{alg:protection} with \(k_{P}=3\) and \(K_{P}=4\) to the prolongation, and \Cref{alg:detection} with \(K=3\) to all other operations.
 It is seen that the strategy has resulted in the factor \(\c{O}\left(\sqrt{n_{L}\varepsilon^{2}}\right)\) in the Lyapunov spectral radius being replaced by \(\c{O}\left(n_{L}\varepsilon^{\alpha}\right)\) with \(\alpha=3\).
\begin{figure}
  \centering
  \includegraphics{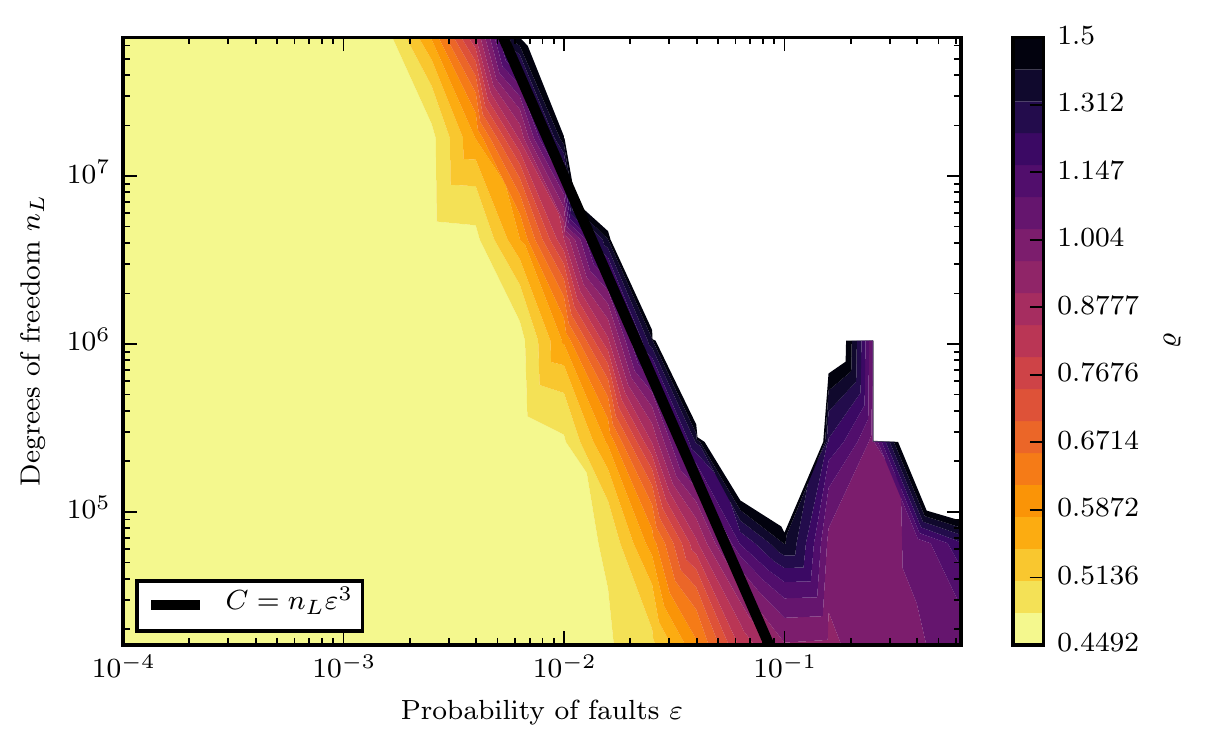}
  \caption{
    Lyapunov spectral radius $\varrho(\randmat{E}_L)$ for the iteration matrix for the Fault-Prone W-Cycle Multigrid Method with \(K_{P}=4\) replicas and acceptance threshold \(k_{P}=3\) in the prolongation and \(K=3\) replicas in all other operations.
  } \label{fig:poisson2dgeoSoftWcycle4-qd-rhoL}
\end{figure}

Would a cheaper detection and protection strategy (e.g. with \(K_{P}=3\), \(k_{P}=2\) and \(K=2\)) suffice?
\Cref{fig:softFaultCurves} shows the results obtained employing a variety of different choices for \(K\), \(K_{P}\) and \(k_{P}\).
It is seen that the cheapest strategy that results in the growth \(\c{O}\left(\sqrt{n_{L}\varepsilon^{2}}\right)\) being replaced by \(\c{O}\left(n_{L}\varepsilon^{\alpha}\right)\), \(\alpha>2\), is indeed \(K_{P}=4\), \(k_{P}=3\) and \(K=3\).

\begin{figure}
  \centering
  \includegraphics{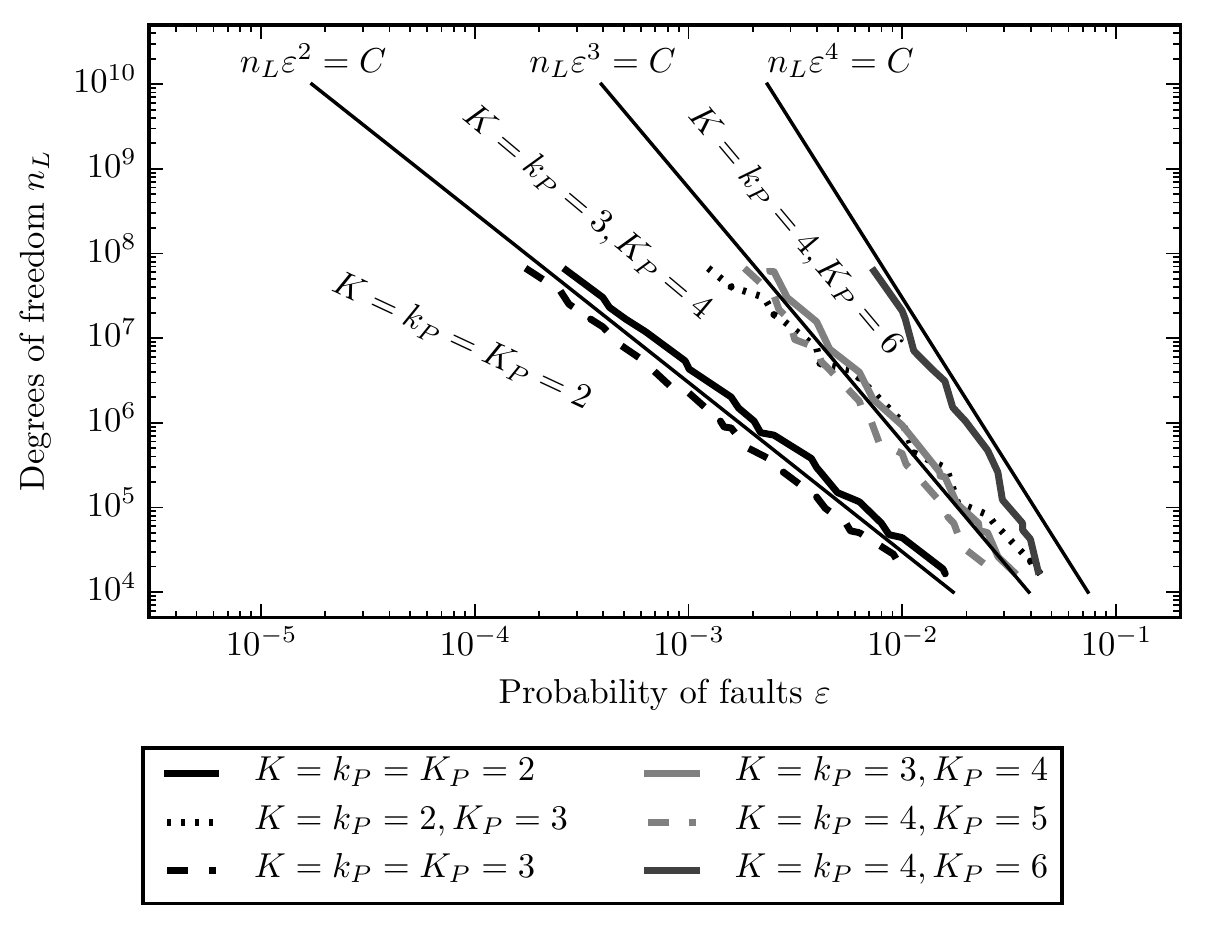}
  \caption{
    Level sets $\varrho(\randmat{E}_L)=0.57$ (50\% more iterations than the ideal fault-free case) for different levels of detection and protection in the case of the 2D Poisson problem.
    Optimal parameter combinations are marked in thick solid lines.
    The obtained optimal detection and protection parameters for each region in the \(\varepsilon\)-\(n_{L}\)-plane are marked in the plot.
  } \label{fig:softFaultCurves}
\end{figure}

Moreover, if one wishes to obtain a given rate \(\alpha\geq2\), then one can choose \(K=k_{P}=\alpha\) and \(K_{P}=2\alpha-2\).
As problems in three spatial dimensions are more resilient to the effect of laissez-faire mitigation, as reflected by \cref{eq:TGbound} and the numerical evidence in \Cref{sec:main-results}, we expect the strategy \(K=K_{P}=k_{P}=\alpha\) to be sufficient for \(2\leq\alpha\leq 6\).

\section{Conclusion}

In this work, we extended previous results concerning the convergence of the Fault-Prone Two Grid Method to the multigrid case, and showed that, if the prolongation is protected against faults, then the resulting Fault-Prone Multigrid Method is resilient.
Numerical examples illustrated this result also holds in a variety of settings where the theory does not apply, along with the necessity of protecting the prolongation.
Finally, we presented one possible simple fault detection strategy and showed that the resulting algorithm is also resilient.

\FloatBarrier

\appendix

\section{Proof of \texorpdfstring{\Cref{thm:MGNoProlong}}{Theorem 3}}
\label{sec:proof}

Throughout the appendices, \(C\) will be a generic constant whose value can change from line to line, but which is independent of \(\ell\) and \(\varepsilon\).

\begin{definition}[Energy norms]
  For matrices \(\mat{Z}\in\RR^{n_{\ell}\times n_{k}}\), we define the energy norm \(\norm{\mat{Z}}_{A}\) and the double energy norm \(\norm{\mat{Z}}_{A^{2}}\) as
  \begin{align*}
    \norm{\mat{Z}}_{A} &= \norm{\mat{A}_{\ell}^{\frac{1}{2}} \mat{Z} \mat{A}_{k}^{-\frac{1}{2}}}_{2},
    & \norm{\mat{Z}}_{A^{2}} &= \norm{\mat{A}_{\ell} \mat{Z} \mat{A}_{k}^{-1}}_{2}.
  \end{align*}
  For matrices \(\mat{W}\in\RR^{n_{\ell}^{2}\times n_{k}^{2}}\), we define the tensor energy norm \(\norm{\mat{W}}_{A}\) and the tensor double energy norm \(\norm{\mat{W}}_{A^{2}}\) as
  \begin{align*}
    \norm{\mat{W}}_{A} &= \norm{\left(\mat{A}_{\ell}^{\frac{1}{2}}\right)^{\otimes2} \mat{Z} \left(\mat{A}_{k}^{-\frac{1}{2}}\right)^{\otimes2}}_{2},
    & \norm{\mat{W}}_{A^{2}} &= \norm{\mat{A}_{\ell}^{\otimes2} \mat{Z} \left(\mat{A}_{k}^{-1}\right)^{\otimes2}}_{2}.
  \end{align*}
  In all cases \(\norm{\bullet}_{2}\) is the spectral norm.
\end{definition}

We set
\begin{align*}
  \Exp{\randmat{X}_{\ell}^{(\bullet)}} &= e_{\ell}^{(\bullet)}\mat{I}, & \Var{\randmat{X}_{\ell}^{(\bullet)}} &= \mat{V}_{\ell}^{(\bullet)}.
\end{align*}
and expand
\begin{align*}
  \Exp{\left(\randmat{E}_{\ell}^{S}\right)^{\otimes2}} &= \Exp{\randmat{E}_{\ell}^{S}}^{\otimes2} + \mat{C}_{\ell}^{(S)}, \\
  \mat{C}_{\ell}^{(S)} &= \mat{V}_{\ell}^{(S)} \left(\mat{N}_{\ell}\mat{A}_{\ell}\right)^{\otimes2},
\end{align*}
using Lemma 10 in \cite{AinsworthGlusa216_IsMultigridMethodFaultTolerant}.
Similar to the fault-free case, we derive recursive inequalities for the iteration matrix of the multigrid method.

\begin{theorem}
  If \cref{as:smoothing,as:approximation,as:smoother,as:prolongation,as:faults} hold, then the iteration matrix of the Fault-Prone Multigrid Method satisfies the following recursive inequalities:
  \begin{align}
    \opnorm{\Exp{\randmat{E}_{\ell}\left(\nu_{1},\nu_{2},\gamma\right)}}
    & \leq \opnorm{\Exp{\randmat{E}_{\ell}^{TG}\left(\nu_{1},\nu_{2}\right)}} + C\varepsilon + C_{*}\opnorm{\Exp{\randmat{E}_{l-1}\left(\nu_{1},\nu_{2},\gamma\right)}}^{\gamma}, \label{eq:17}
    \intertext{and}
    \opnorm{\Exp{\randmat{E}_{\ell}\left(\nu_{1},\nu_{2},\gamma\right)^{\otimes2}}}
    & \leq \opnorm{\Exp{\randmat{E}_{\ell}^{TG}\left(\nu_{1},\nu_{2}\right)^{\otimes2}}} + C\varepsilon \label{eq:1}\\
    & \quad +2C_{*}\opnorm{\Exp{\randmat{E}_{\ell}^{TG}\left(\nu_{1},\nu_{2}\right)}} \opnorm{\Exp{\randmat{E}_{l-1}\left(\nu_{1},\nu_{2},\gamma\right)}}^{\gamma} \nonumber \\
    & \quad + C_{*}^{2} \opnorm{\Exp{\randmat{E}_{l-1}\left(\nu_{1},\nu_{2},\gamma\right)^{\otimes2}}}^{\gamma}, \nonumber
    \intertext{with}
    C_{*} & = C_{S}\underline{C}_{p}\overline{C}_{p}\left(C_{S}+C_{A}\eta(\nu_{2})\right). \nonumber
  \end{align}
  The constants \(C\) and \(C_{*}\) are independent of the level \(\ell\) and \(\varepsilon\).
\end{theorem}

If \(\varepsilon=0\), the second inequality matches exactly the recursive inequality that we would use for the fault-free case, while the first inequality matches the tensor square power of the same inequality.

\begin{proof}
  Since the proof on level \(\ell\) only involves the levels \(\ell\) and \(\ell-1\), we will drop the first subscript and replace the second one with a subscript \(C\).

  We can write the iteration matrix as
  \begin{align*}
    \randmat{E}\left(\nu_{1},\nu_{2},\gamma\right)
    & = \left(\randmat{E}^{S,\text{post}}\right)^{\nu_{2}}\left(\mat{I}-\mat{P}\left(\mat{I}-\left(\randmat{E}_{C}\right)^{\gamma}\right)\mat{A}_{C}^{-1}\randmat{X}^{(R)}_{C} \mat{R}\randmat{X}^{(\rho)} \mat{A}\right) \left(\randmat{E}^{S,\text{pre}}\right)^{\nu_{1}} \\
    & =\underbrace{\left(\randmat{E}^{S,\text{post}}\right)^{\nu_{2}}\left(\mat{I}-\mat{P}\mat{A}_{C}^{-1}\randmat{X}^{(R)}_{C} \mat{R}\randmat{X}^{(\rho)} \mat{A}\right) \left(\randmat{E}^{S,\text{pre}}\right)^{\nu_{1}}}_{\randmat{E}^{TG}\left(\nu_{1},\nu_{2}\right)} \\
    &\quad+ \underbrace{\left(\randmat{E}^{S,\text{post}}\right)^{\nu_{2}}\mat{P}\left(\randmat{E}_{C}\right)^{\gamma}\mat{A}_{C}^{-1}\randmat{X}^{(R)}_{C} \mat{R}\randmat{X}^{(\rho)} \mat{A}\left(\randmat{E}^{S,\text{pre}}\right)^{\nu_{1}}}_{=:\randmat{C}}.
  \end{align*}
  Since
  \begin{align*}
    \opnorm{\Exp{\left(\randmat{E}^{S}\right)^{\nu_{2}}\mat{P}}}
    & = \norm{\mat{A}\Exp{\randmat{E}^{S}}^{\nu_{2}}\mat{P}\mat{A}_{C}^{-1}}_{2}  \\
    & = \norm{\mat{A}_{C}^{-1}\mat{R}\mat{A}\Exp{\randmat{E}^{S}}^{\nu_{2}}}_{2}   \\
    & \leq \underline{C}_{p} \left(\norm{\Exp{\randmat{E}^{S}}^{\nu_{2}}}_{2}+\norm{\mat{E}^{CG}\Exp{\randmat{E}^{S}}^{\nu_{2}}}_{2}\right)  \\
    & \leq \underline{C}_{p} \left(C_{S}+C_{A}\eta(\nu_{2})\right) + C\varepsilon,  \\
    \opnorm{\Exp{\mat{A}_{C}^{-1}\randmat{X}_{C}^{(R)}\mat{R}\randmat{X}^{(\rho)}\mat{A}}}
    &\leq \left|e_{C}^{(R)}e^{(\rho)}\right| \norm{\mat{R}}_{2}
    \leq \overline{C}_{p} + C\varepsilon, \\
    \opnorm{\Exp{\randmat{E}^{S}}^{\nu_{1}}} &\leq C_{S} + C\varepsilon,
  \end{align*}
  we have by \cref{as:prolongation,as:faults} that
  \begin{align}
    \opnorm{\Exp{\randmat{C}}}
    & \leq \opnorm{\Exp{\randmat{E}^{S}}^{\nu_{2}}\mat{P}} \opnorm{\Exp{\randmat{E}_{C}}}^{\gamma} \opnorm{\Exp{\mat{A}_{C}^{-1}\randmat{X}_{C}^{(R)}\mat{R}\randmat{X}^{(\rho)}\mat{A}}} \opnorm{\Exp{\randmat{E}^{S}}^{\nu_{1}}} \nonumber \\
    & \leq C_{*}\opnorm{\Exp{\randmat{E}_{C}}}^{\gamma} + C\varepsilon. \label{eq:8}
  \end{align}
  Therefore the first inequality of the theorem follows from
  \begin{align*}
    \opnorm{\Exp{\randmat{E}}} & \leq \opnorm{\Exp{\randmat{E}^{TG}}} + \opnorm{\Exp{\randmat{C}}}.
  \end{align*}

  The second inequality is more work, but follows the same path:
  \begin{align}
    \opnorm{\Exp{\randmat{E}^{\otimes2}}}
    & \leq \opnorm{\Exp{\left(\randmat{E}^{TG}\right)^{\otimes2}}} + \opnorm{\Exp{\randmat{C}^{\otimes2}}}
      +2\opnorm{\Exp{\randmat{E}^{TG} \otimes \randmat{C}}}. \label{eq:7}
  \end{align}
  We start by bounding the second term on the right-hand side.
  \begin{align}
    \opnorm{\Exp{\randmat{C}^{\otimes2}}}
    \leq & \opnorm{\Exp{\left(\left(\randmat{E}^{S}\right)^{\nu_{2}}\mat{P}\right)^{\otimes2}}}
      \opnorm{\Exp{\randmat{E}_{C}^{\otimes2}}}^{\gamma} \label{eq:4} \\
    & \times \opnorm{\Exp{\left(\mat{A}_{C}^{-1}\randmat{X}^{(R)}_{C} \mat{R}\randmat{X}^{(\rho)} \mat{A}\right)^{\otimes2} }}
      \opnorm{\Exp{\left(\left(\randmat{E}^{S}\right)^{\nu_{1}}\right)^{\otimes2}}} \nonumber
  \end{align}
  We have by
  \begin{align*}
    \opnorm{\Exp{\left(\left(\randmat{E}^{S}\right)^{\nu_{2}}\mat{P}\right)^{\otimes2}}}
    & \leq \opnorm{\left(\Exp{\randmat{E}^{S}}^{\nu_{2}}\mat{P}\right)^{\otimes2}} + C\opnorm{\mat{C}^{(S)}} \numberthis \label{eq:6} \\
    & = \norm{\mat{A}\Exp{\randmat{E}^{S}}^{\nu_{2}}\mat{P}\mat{A}_{C}^{-1}}_{2}^{2} + C\varepsilon  \\
    & = \norm{\mat{A}_{C}^{-1}\mat{R}\mat{A}\Exp{\randmat{E}^{S}}^{\nu_{2}}}_{2}^{2} + C\varepsilon  \\
    & \leq \underline{C}_{p}^{2}\left(\norm{\Exp{\randmat{E}^{S}}^{\nu_{2}}}_{2}+\norm{\mat{E}^{CG}\Exp{\randmat{E}^{S}}^{\nu_{2}}}_{2}\right)^{2} + C\varepsilon  \\
    & \leq \underline{C}_{p}^{2}\left(C_{S}+C_{A}\eta(\nu_{2})\right)^{2} + C\varepsilon.
  \end{align*}
  Since by \cref{as:prolongation,as:faults}
  \begin{align*}
    \opnorm{\Exp{\left(\mat{A}_{C}^{-1}\randmat{X}_{C}^{(R)}\mat{R}\randmat{X}^{(\rho)}\mat{A}\right)^{\otimes2}}}
    & = \norm{\Exp{\left(\randmat{X}_{C}^{(R)}\mat{R}\randmat{X}^{(\rho)}\right)^{\otimes2}}}_{2}  \numberthis \label{eq:22}\\
    & \leq \norm{\Exp{\left(\randmat{X}_{C}^{(R)}\right)^{\otimes2}}}_{2}\norm{\mat{R}}_{2}^{2}\norm{\Exp{\left(\randmat{X}^{(\rho)}\right)^{\otimes2}}}_{2} \\
    & \leq \left(\norm{\Var{\randmat{X}_{C}^{(R)}}}_{2} + \norm{\Exp{\randmat{X}_{C}^{(R)}}}_{2}^{2}\right)\norm{\mat{R}}_{2}^{2} \\
    & \qquad \times\left(\norm{\Var{\randmat{X}^{(\rho)}}}_{2} + \norm{\Exp{\randmat{X}^{(\rho)}}}_{2}^{2}\right) \\
    & \leq \left(\varepsilon+\left(e^{(R)}\right)^{2}\right)\left(\varepsilon+\left(e^{(\rho)}\right)^{2}\right)\norm{\mat{R}}_{2}^{2} \\
    & \leq \overline{C}_{p}^{2} + C\varepsilon.
  \end{align*}
  and
  \begin{align}
    \opnorm{\Exp{\left(\left(\randmat{E}^{S}\right)^{\nu_{1}}\right)^{\otimes2}}}
    & = \opnorm{\left(\Exp{\randmat{E}^{S}}^{\otimes2}+\mat{C}^{(S)}\right)^{\nu_{1}}} \label{eq:3} \\
    &\leq \opnorm{\Exp{\randmat{E}^{S}}^{\nu_{1}}}^{2} + C\varepsilon
     \leq C_{S}^{2} + C\varepsilon, \nonumber
  \end{align}
  Inserting \cref{eq:6,eq:22,eq:3} into \cref{eq:4}, we have
  \begin{align}
    \opnorm{\Exp{\randmat{C}^{\otimes2}}} & \leq C_{*}^{2}\opnorm{\Exp{\randmat{E}_{C}^{\otimes2}}}^{\gamma} + C\varepsilon. \label{eq:16}
  \end{align}
  Now we turn to the third term of \cref{eq:7}.
  We split
  \begin{align*}
    \randmat{E}^{TG}
    & = \left(\randmat{E}^{S,\text{post}}\right)^{\nu_{2}} \left(\mat{I}-\mat{P}\mat{A}_{C}^{-1}\randmat{X}^{(R)}_{C} \mat{R}\randmat{X}^{(\rho)} \mat{A}\right) \left(\randmat{E}^{S,\text{pre}}\right)^{\nu_{1}}\\
    & = \Exp{\randmat{E}^{S}}^{\nu_{2}} \mat{E}^{CG} \Exp{\randmat{E}^{S}}^{\nu_{1}} \\
    &\quad + \left(\left(\randmat{E}^{S,\text{post}}\right)^{\nu_{2}}-\Exp{\randmat{E}^{S}}^{\nu_{2}}\right) \mat{E}^{CG} \Exp{\randmat{E}^{S}}^{\nu_{1}}\\
    &\quad + \Exp{\randmat{E}^{S}}^{\nu_{2}} \mat{E}^{CG} \left(\left(\randmat{E}^{S,\text{pre}}\right)^{\nu_{1}}-\Exp{\randmat{E}^{S}}^{\nu_{1}}\right)\\
    &\quad + \left(\left(\randmat{E}^{S,\text{post}}\right)^{\nu_{2}}-\Exp{\randmat{E}^{S}}^{\nu_{2}}\right) \mat{E}^{CG} \left(\left(\randmat{E}^{S,\text{pre}}\right)^{\nu_{1}}-\Exp{\randmat{E}^{S}}^{\nu_{1}}\right)\\
    &\quad + \left(\randmat{E}^{S,\text{post}}\right)^{\nu_{2}} \mat{P}\left(\mat{A}_{C}^{-1}\mat{R}\mat{A}-\mat{A}_{C}^{-1}\randmat{X}^{(R)}_{C} \mat{R}\randmat{X}^{(\rho)} \mat{A}\right) \left(\randmat{E}^{S,\text{pre}}\right)^{\nu_{1}}
  \end{align*}
  so that
  \begin{align*}
    & \opnorm{\Exp{\randmat{E}^{TG} \otimes \randmat{C}}} \label{eq:2} \numberthis \\
    \leq & \opnorm{\Exp{\randmat{E}^{S}}^{\nu_{2}}\mat{E}^{CG}\Exp{\randmat{E}^{S}}^{\nu_{1}}}\opnorm{\Exp{\randmat{C}}}  \\
    & \quad + \opnorm{\mat{E}^{CG}} \opnorm{\mat{P}} \opnorm{\Exp{\randmat{E}_{C}}}^{\gamma} \opnorm{\Exp{\mat{A}_{C}^{-1}\randmat{X}^{(R)}_{C} \mat{R}\randmat{X}^{(\rho)} \mat{A}}} \\
    & \qquad \left\{
      \opnorm{\Exp{\left(\randmat{E}^{S}\right)^{\nu_{2}}\otimes\left(\left(\randmat{E}^{S}\right)^{\nu_{2}}-\Exp{\randmat{E}^{S}}^{\nu_{2}}\right)}}
      \opnorm{\Exp{\randmat{E}^{S}}^{\nu_{1}}}^{2}
      \right. \\
    & \qquad+
      \opnorm{\Exp{\randmat{E}^{S}}^{\nu_{2}}}^{2}
      \opnorm{\Exp{\left(\randmat{E}^{S}\right)^{\nu_{1}}\otimes\left(\left(\randmat{E}^{S}\right)^{\nu_{1}}-\Exp{\randmat{E}^{S}}^{\nu_{1}}\right)}} \\
    & \qquad+
      \opnorm{\Exp{\left(\randmat{E}^{S}\right)^{\nu_{2}}\otimes\left(\left(\randmat{E}^{S}\right)^{\nu_{2}}-\Exp{\randmat{E}^{S}}^{\nu_{2}}\right)}} \\
    & \qquad\left. \times
      \opnorm{\Exp{\left(\randmat{E}^{S}\right)^{\nu_{1}}\otimes\left(\left(\randmat{E}^{S}\right)^{\nu_{1}}-\Exp{\randmat{E}^{S}}^{\nu_{1}}\right)}}
      \right\} \\
    & \quad +
      \opnorm{\Exp{\left(\mat{A}_{C}^{-1}\mat{R}\mat{A}\right)\otimes\left(\mat{A}_{C}^{-1}\randmat{X}^{(R)}_{C} \mat{R}\randmat{X}^{(\rho)} \mat{A}\right) - \left(\mat{A}_{C}^{-1}\randmat{X}^{(R)}_{C} \mat{R}\randmat{X}^{(\rho)} \mat{A}\right)^{\otimes2}}} \\
    & \qquad \times \opnorm{\Exp{\randmat{E}_{C}}}^{\gamma} \opnorm{\Exp{\left(\left(\randmat{E}^{S}\right)^{\nu}\mat{P}\right)^{\otimes2}}}
      \opnorm{\Exp{\left(\left(\randmat{E}^{S}\right)^{\nu_{1}}\right)^{\otimes2}}}
  \end{align*}

  Now, we have by \cref{as:approximation,as:prolongation} that
  \begin{align}
    \opnorm{\mat{E}^{CG}} &=\norm{\mat{E}^{CG}}_{2} \leq C_{A}, \label{eq:10}\\
    \opnorm{\mat{P}}
    & =\norm{\mat{A}\mat{P}\mat{A}_{C}^{-1}}_{2} \leq \underline{C}_{p},   \label{eq:11}
  \end{align}
  and by \Cref{as:faults,as:prolongation} that
  \begin{align*}
    & \opnorm{\Exp{\left(\mat{A}_{C}^{-1}\mat{R}\mat{A}\right)\otimes\left(\mat{A}_{C}^{-1}\randmat{X}^{(R)}_{C} \mat{R}\randmat{X}^{(\rho)} \mat{A}\right) - \left(\mat{A}_{C}^{-1}\randmat{X}^{(R)}_{C} \mat{R}\randmat{X}^{(\rho)} \mat{A}\right)^{\otimes2}}} \label{eq:13} \numberthis \\
    = & \norm{\Exp{\mat{R}\otimes \left(\randmat{X}_{C}^{(R)}\mat{R}\randmat{X}^{(\rho)}\right)-\left(\randmat{X}_{C}^{(R)}\mat{R}\randmat{X}^{(\rho)}\right)^{\otimes2}}}_{2}\\
    \leq & \norm{\mat{R}}_{2}^{2}
           \left\{
           \abs{e_{C}^{(R)}e^{(\rho)}}\abs{1-e_{C}^{(R)}e^{(\rho)}} + \left(e^{(\rho)}\right)^{2}\norm{\Var{\randmat{X}_{C}^{(R)}}}_{2} \right. \\
    & + \left.\left(e_{C}^{(R)}\right)^{2}\norm{\Var{\randmat{X}^{(\rho)}}}_{2} +  \norm{\Var{\randmat{X}_{C}^{(R)}}}_{2} \norm{\Var{\randmat{X}^{(\rho)}}}_{2}
           \right\} \\
    \leq & C\varepsilon
  \end{align*}
  and that
  \begin{align*}
    &\opnorm{\Exp{\left(\randmat{E}^{S}\right)^{\nu}\otimes\left(\left(\randmat{E}^{S}\right)^{\nu}-\Exp{\randmat{E}^{S}}^{\nu}\right)}} \label{eq:32} \numberthis \\
    = & \opnorm{\Exp{\left(\randmat{E}^{S}\right)^{\otimes2}}^{\nu} - \left(\Exp{\randmat{E}^{S}}^{\otimes2}\right)^{\nu}}  \\
  = &\opnorm{\left(\Exp{\randmat{E}^{S}}^{\otimes2} + \mat{C}^{(S)}\right)^{\nu} - \left(\Exp{\randmat{E}^{S}}^{\otimes2}\right)^{\nu}} \\
  \leq & C\opnorm{\mat{C}^{(S)}} \leq C\varepsilon.
  \end{align*}
  Also,
  \begin{align*}
    & \opnorm{\Exp{\randmat{E}^{S}}^{\nu_{2}}\mat{E}^{CG}\Exp{\randmat{E}^{S}}^{\nu_{1}}} \label{eq:15} \numberthis \\
    \leq & \opnorm{\Exp{\randmat{E}^{TG}\left(\nu_{1},\nu_{2}\right)}}
      + \abs{1-e_{C}^{(R)}}\abs{1-e^{(\rho)}} \norm{\Exp{\randmat{E}^{S}}^{\nu_{1}}(\mat{I}-\mat{E}^{CG})\Exp{\randmat{E}^{S}}^{\nu_{2}}}_{2} \\
     = &\opnorm{\Exp{\randmat{E}^{TG}\left(\nu_{1},\nu_{2}\right)}} + C\varepsilon.
  \end{align*}
  Hence we find by inserting \cref{eq:10,eq:11,eq:13,eq:15,eq:8,eq:32} into \cref{eq:2}
  \begin{align}
    \opnorm{\Exp{\randmat{E}^{TG} \otimes \randmat{C}}}
    & \leq C_{*}\opnorm{\Exp{\randmat{E}^{TG}}}\opnorm{\Exp{\randmat{E}_{C}}}^{\gamma} + C\varepsilon. \label{eq:9}
  \end{align}
  Therefore, the second recursive inequality
  \begin{align*}
    \opnorm{\Exp{\randmat{E}^{\otimes2}}}
    & \leq \opnorm{\Exp{\left(\randmat{E}^{TG}\right)^{\otimes2}}} + C\varepsilon \\
    & \quad +2C_{*}\opnorm{\Exp{\randmat{E}^{TG}}}\opnorm{\Exp{\randmat{E}_{C}}}^{\gamma} \\
    & \quad + \left(C_{*}\right)^{2} \opnorm{\Exp{\randmat{E}_{C}^{\otimes2}}}^{\gamma}
  \end{align*}
  follows from \cref{eq:9,eq:16,eq:7}.
\end{proof}

We adapt \Cref{lem:reccursion} to the tensor space setting:
\begin{lemma}\label{lem:reccursion2}
  Let the elements of the sequences \(\left\{\eta_{\ell}\right\}_{\ell\geq1}\) and \(\left\{\eta_{\ell,\otimes2}\right\}_{\ell\geq1}\) be bounded as follows
  \begin{align*}
    \eta_{1} & \leq \xi, &
    \eta_{\ell} & \leq \xi + C_{*}\eta_{\ell-1}^{\gamma}, \quad \ell\geq 2, \\
    \eta_{1,\otimes2}^{2} & \leq \xi_{\otimes2}^{2}, &
    \eta_{\ell,\otimes2}^{2} & \leq \xi_{\otimes2}^{2} + 2\xi C_{*}\eta_{\ell-1}^{\gamma} + C_{*}^{2}\eta_{\ell-1,\otimes2}^{2\gamma}, \quad \ell\geq 2.
  \end{align*}
  If \(\gamma\geq 2\), \(C_{*}\gamma > 1\) and \(\max\left\{\xi,\xi_{\otimes2}\right\} \leq \frac{\gamma-1}{\gamma}\frac{1}{\left(C_{*}\gamma\right)^{\frac{1}{\gamma-1}}}\), then
  \begin{align*}
    \max\left\{\eta_{\ell}, \eta_{\ell,\otimes2}\right\} & \leq\frac{\gamma}{\gamma-1}\max\left\{\xi,\xi_{\otimes2}\right\}<1, & \ell\geq 1.
    \intertext{In the case of \(\gamma=2\), we have}
    \max\left\{\eta_{\ell}, \eta_{\ell,\otimes2}\right\} & \leq \frac{2\max\left\{\xi,\xi_{\otimes2}\right\}}{1+\sqrt{1-4C_{*}\max\left\{\xi,\xi_{\otimes2}\right\}}}<1, & \ell\geq 1.
  \end{align*}
\end{lemma}
\begin{proof}
  Set \(\alpha_{\ell}:=\max\left\{\eta_{\ell},\eta_{\ell,\otimes2}\right\}\) and \(\beta:=\max\left\{\xi,\xi_{\otimes2}\right\}\).
  Then
  \begin{align*}
    \eta_{1}              & \leq \beta, &
    \eta_{\ell}              & \leq \beta + C_{*}\alpha_{\ell-1}^{\gamma}, \quad\ell\geq 2, \\
    \eta_{1,\otimes2}^{2} & \leq \beta^{2}, &
    \eta_{\ell,\otimes2}^{2} & \leq \beta^{2} + 2\beta C_{*}\alpha_{\ell-1}^{\gamma} + C_{*}^{2}\alpha_{\ell-1}^{2\gamma}, \quad \ell\geq 2,
  \end{align*}
  and hence
  \begin{align*}
    \alpha_{1} & \leq\beta, &
    \alpha_{\ell} & \leq \beta + C_{*}\alpha_{\ell-1}^{\gamma}, \quad \ell\geq2.
  \end{align*}
  Application of \cref{lem:reccursion} gives for \(\gamma\geq 2\)
  \begin{align*}
    \alpha_{\ell} & \leq \frac{\gamma}{\gamma-1}\beta<1 \\
    \intertext{and for \(\gamma=2\)}
    \alpha_{\ell} & \leq \frac{2\beta}{1+\sqrt{1-4C_{*}\beta}}<1.
  \end{align*}
\end{proof}

\MGNoProlong*
\begin{proof}
  Set
  \begin{align*}
    \eta_{\ell} & :=\opnorm{\Exp{\randmat{E}_{\ell}\left(\nu_{1},\nu_{2},\gamma\right)}}, \\
    \eta_{\ell,\otimes2} & :=\opnorm{\Exp{\randmat{E}_{\ell}\left(\nu_{1},\nu_{2},\gamma\right)^{\otimes2}}}^{\frac{1}{2}}, \\
    \xi & :=\max_{\ell}\norm{\mat{E}_{\ell}^{TG}\left(\nu_{2},\nu_{1}\right)}_{2} + C\varepsilon, \\
    \xi_{\otimes2} & :=\max_{\ell}\norm{\mat{E}_{\ell}^{TG}\left(\nu_{2},\nu_{1}\right)}_{2} + C\varepsilon.
  \end{align*}
  It follows from
  \begin{align*}
    \Exp{\randmat{E}^{S}_{\ell}} &= \mat{E}^{S}_{\ell} + \left(1-e_{\ell}^{(S)}\right)\left(\mat{I}-\mat{E}^{S}_{\ell}\right), \\
    \Exp{\randmat{E}^{CG}_{\ell}} &= \mat{E}^{CG}_{\ell} + \left(1-e_{\ell-1}^{(R)}e_{\ell}^{(\rho)}\right)\left(\mat{I}-\mat{E}^{CG}_{\ell}\right),
  \end{align*}
  and \Cref{as:faults,as:smoother,as:approximation} that
  \begin{align*}
    \opnorm{\Exp{\randmat{E}_{\ell}^{TG}\left(\nu_{1},\nu_{2}\right)}} & \leq \norm{\mat{E}_{\ell}^{TG}\left(\nu_{2},\nu_{1}\right)}_{2} + C\varepsilon,
  \end{align*}
  and therefore with \cref{eq:17} that \(\eta_{\ell}\leq \xi + C_{*}\eta_{\ell-1}^{\gamma}\).
  By \Cref{thm:TGNoProlong} and \cref{eq:1}, we find \(\eta_{\ell,\otimes2}^{2}\leq \xi_{\otimes2}^{2} + 2\xi C_{*}\eta_{\ell-1}^{\gamma}+C_{*}^{2}\eta_{\ell-1,\otimes2}^{2}\).
  Since \(C_{*}\), \(\xi\) and \(\xi_{\otimes2}\) decrease as the number of smoothing steps increases, \Cref{lem:reccursion2} can be applied for a sufficient number of smoothing steps.
\end{proof}

\section*{Acknowledgments}
This research used resources of the Oak Ridge Leadership Computing Facility at the Oak Ridge National Laboratory, which is supported by the Office of Science of the U.S. Department of Energy under Contract No. DE-AC05-00OR22725.

\bibliography{/home/cag/org/papers}

\begin{thebibliography}{10}

\bibitem{AinsworthGlusa216_IsMultigridMethodFaultTolerant}
{\sc M.~Ainsworth and C.~Glusa}, {\em {Is the multigrid method fault tolerant?
  The Two Grid Case.}}, {SIAM Journal on Scientific Computing},  (Submitted).

\bibitem{AvizienisLaprieEtAl2004_BasicConceptsTaxonomyDependableSecureComputing}
{\sc A.~Avi{\v{z}}ienis, J.-C. Laprie, B.~Randell, and C.~Landwehr}, {\em Basic
  concepts and taxonomy of dependable and secure computing}, IEEE Transactions
  on Dependable and Secure Computing, 1 (2004), pp.~11--33.

\bibitem{BankShermanEtAl1983_RefinementAlgorithmsDataStructures}
{\sc R.~E. Bank, A.~H. Sherman, and A.~Weiser}, {\em Some refinement algorithms
  and data structures for regular local mesh refinement}, Scientific Computing,
  Applications of Mathematics and Computing to the Physical Sciences, 1 (1983),
  pp.~3--17.

\bibitem{BougerolLacroix1985_ProductsRandomMatricesWith}
{\sc P.~Bougerol and J.~Lacroix}, {\em Products of random matrices with
  applications to {S}chr{\"o}dinger operators}, vol.~8 of Progress in
  Probability and Statistics, Birkh{\"a}user Boston Inc., Boston, MA, 1985.

\bibitem{Braess2007_FiniteElements}
{\sc D.~Braess}, {\em {Finite elements. Theory, fast solvers and applications
  in solid mechanics. Translated from German by Larry L. Schumaker. 3rd ed.}},
  {Cambridge: Cambridge University Press}, 2007,
  \url{http://dx.doi.org/10.1017/CBO9780511618635}.

\bibitem{CappelloGeistEtAl2009_TowardExascaleResilience}
{\sc F.~Cappello, A.~Geist, B.~Gropp, L.~Kale, B.~Kramer, and M.~Snir}, {\em
  Toward exascale resilience}, International Journal of High Performance
  Computing Applications,  (2009).

\bibitem{CappelloGeistEtAl2014_TowardExascaleResilience}
{\sc F.~Cappello, A.~Geist, W.~Gropp, S.~Kale, B.~Kramer, and M.~Snir}, {\em
  Toward exascale resilience: 2014 update}, Supercomputing frontiers and
  innovations, 1 (2014), pp.~5--28.

\bibitem{CrisantiPaladinEtAl1993_ProductsRandomMatrices}
{\sc A.~Crisanti, G.~Paladin, and A.~Vulpiani}, {\em Products of random
  matrices}, Springer, 1993.

\bibitem{Doerfler1996_ConvergentAdaptiveAlgorithmPoissonsEquation}
{\sc W.~D{\"o}rfler}, {\em A convergent adaptive algorithm for {P}oisson's
  equation}, SIAM Journal on Numerical Analysis, 33 (1996), pp.~1106--1124.

\bibitem{ErnGuermond2004_TheoryPracticeFiniteElements}
{\sc A.~Ern and J.-L. Guermond}, {\em {Theory and Practice of Finite
  Elements.}}, {Applied Mathematical Sciences 159. New York, NY: Springer},
  2004.

\bibitem{Hackbusch1985_MultiGridMethodsApplications}
{\sc W.~Hackbusch}, {\em Multi-grid methods and applications}, vol.~4,
  Springer-Verlag Berlin, 1985.

\bibitem{Hackbusch1994_IterativeSolutionLargeSparseSystemsEquations}
{\sc W.~Hackbusch}, {\em Iterative solution of large sparse systems of
  equations}, vol.~95 of Applied Mathematical Sciences, Springer-Verlag, New
  York, 1994, \url{http://dx.doi.org/10.1007/978-1-4612-4288-8}.

\bibitem{HeraultRobert2015_FaultToleranceTechniquesHighPerformanceComputing}
{\sc T.~Herault and Y.~Robert}, {\em {Fault-Tolerance Techniques for
  High-Performance Computing}}, Springer, 2015.

\bibitem{KarypisKumar1998_FastHighQualityMultilevel}
{\sc G.~Karypis and V.~Kumar}, {\em A fast and high quality multilevel scheme
  for partitioning irregular graphs}, SIAM Journal on Scientific Computing, 20
  (1998), pp.~359--392, \url{http://dx.doi.org/10.1137/S1064827595287997}.

\bibitem{KroegerPreusser2008_Stability8TetrahedraShortest}
{\sc T.~Kr{\"o}ger and T.~Preusser}, {\em Stability of the 8-tetrahedra
  shortest-interior-edge partitioning method}, Numerische Mathematik, 109
  (2008), pp.~435--457.

\bibitem{SnirWisniewskiEtAl2014_AddressingFailuresExascaleComputing}
{\sc M.~Snir, R.~W. Wisniewski, J.~A. Abraham, S.~V. Adve, S.~Bagchi,
  P.~Balaji, J.~Belak, P.~Bose, F.~Cappello, B.~Carlson, et~al.}, {\em
  Addressing failures in exascale computing}, International Journal of High
  Performance Computing Applications, 28 (2014), pp.~129--173.

\bibitem{TrottenbergOosterleeEtAl2001_Multigrid}
{\sc U.~Trottenberg, C.~W. Oosterlee, and A.~Sch{\"u}ller}, {\em Multigrid},
  Academic Press Inc., San Diego, CA, 2001.
\newblock With contributions by A. Brandt, P. Oswald and K. St{\"u}ben.

\end{thebibliography}
\bibliographystyle{siamplain}

\end{document}